\documentclass{amsart}

\usepackage{graphicx}
\usepackage{tikz}
\usetikzlibrary{matrix,arrows,decorations.pathmorphing}
\usepackage{tikz-cd}
\usepackage{hyperref}
\usepackage{amsthm}
\usepackage[mathscr]{euscript}
\usepackage{amssymb}
\newtheorem{theorem}{Theorem}[section]
\newtheorem{lemma}[theorem]{Lemma}

\theoremstyle{definition}

\newtheorem{conjecture}[theorem]{Conjecture}

\newtheorem{proposition}[theorem]{Proposition}
\newtheorem{corollary}[theorem]{Corollary}

\newtheorem{subtheorem}{Theorem}[theorem]

\theoremstyle{remark}
\newtheorem{remark}[theorem]{Remark}

\numberwithin{equation}{section}

\newcommand{\RomanNumeralCaps}[1]
    {\MakeUppercase{\romannumeral #1}}
\begin{document}

\title{The Topological Period-Index Problem over $8$-complexes, \RomanNumeralCaps{1}}

\author{Xing Gu}
\address{School of Mathematics and Statistics, the University of Melbourne, Parkville VIC 3010, Australia}
\curraddr{}
\email{xing.gu@unimelb.edu.au}
\thanks{The author acknowledges support from the Australian Research Council via ARC-DP170102328.}


\subjclass[2000]{}

\date{}

\dedicatory{}

\keywords{Brauer groups, twisted K-theory, period-index problems.}

\begin{abstract}
 We study the Postnikov tower of the classifying space of a compact Lie group $P(n,mn)$, which gives obstructions to lifting a topological Brauer class of period $n$ to a $PU_{mn}$-torsor, where the base space is a CW complex of dimension $8$. Combined with the study of a twisted version of Atiyah-Hirzebruch spectral sequence, this solves the topological period-index problem for CW complexes of dimension $8$.
\end{abstract}

\maketitle
\section{Introduction}
This paper is a sequel to \cite{An} and \cite{An1}, in which Antieau and Williams initiated the study of the topological period-index problem. Given a path-connected topological space $X$, let $\operatorname{Br}(X)$ be the topological Brauer group defined in \cite{An1}, whose underlying set is the Azumaya algebras modulo the Brauer equivalence: $\mathscr{A}_0$ and $\mathscr{A}_1$ are called Brauer equivalent if there are vector bundles $\mathscr{E}_0$ and $\mathscr{E}_1$ such that
$$\mathscr{A}_0\otimes\operatorname{End}(\mathscr{E}_0)\cong \mathscr{A}_1\otimes\operatorname{End}(\mathscr{E}_1).$$
The multiplication is given by tensor product.

\begin{remark}
In its full generality, Brauer group can be defined for any locally-ringed topos. See \cite{An4}, for example.
\end{remark}
Azumaya algebras over $X$ of degree $r$ are classified by the collection of $PU_r$-torsors over $X$, i.e., the cohomology set $H^{1}(X;PU_r)$, where $PU_r$ is the projective unitary group of degree $r$. Consider the short exact sequences of Lie groups
\begin{equation}\label{ses1}
1\rightarrow S^1\rightarrow U_r\rightarrow PU_r\rightarrow 1
\end{equation}
and
\begin{equation}\label{ses2}
0\rightarrow\mathbb{Z}\rightarrow\mathbb{C}\xrightarrow{\textrm{exp}} S^1\rightarrow 1
\end{equation}
where the arrow $S^1\rightarrow U_r$ is the inclusion of scalars. Then the composition of Bockstein homomorphisms
\begin{equation}\label{Bockstein}
H^{1}(X;PU_r)\rightarrow H^{2}(X;S^1)\rightarrow H^{3}(X;\mathbb{Z})
\end{equation}
associates an Azumaya algebra $\mathscr{A}$ to a class $\alpha\in H^{3}(X;\mathbb{Z})$. The exactness of the sequences above implies that
\begin{enumerate}
\item $\alpha\in H^{3}(X;\mathbb{Z})_{\textrm{tor}}$, the subgroup of torsion elements of $ H^{3}(X;\mathbb{Z})$, and
\item the class $\alpha$ only depends on the Brauer equivalence class of $\mathscr{A}$.
\end{enumerate}
Therefore, we established a function $\textrm{Br}(X)\rightarrow H^{3}(X;\mathbb{Z})_{\textrm{tor}}$. It is not hard to show (\cite{Gr}) that this function is in fact an inclusion of a subgroup. For this reason, $H^{3}(X;\mathbb{Z})_{\textrm{tor}}$ is also called the cohomological Brauer group of $X$, and is sometimes denoted by $\textrm{Br}'(X)$.

Serre showed (\cite{Gr}) that when $X$ is a finite CW complex, the inclusion is also surjective. Hence, for any $\alpha\in H^{3}(X;\mathbb{Z})_{\textrm{tor}}$, there is some $r$ such that a $PU_r$-torsor over $X$ is associated to $\alpha$ via the homomorphism (\ref{Bockstein}). Let $\operatorname{per}(\alpha)$ denote the order of $\alpha$ as an element of the group $H^{3}(X;\mathbb{Z})$, then Serre also showed (\cite{Gr}) that $\operatorname{per}(\alpha)|r$, for all $r$ such that there is a $PU_r$-torsor over $X$ associated to $\alpha$ in the way described above. Let $\operatorname{ind}(\alpha)$ denote the greatest common divisor of all such $r$, then in particular we have
\begin{equation}\label{per div ind}
\operatorname{per}(\alpha)|\operatorname{ind}(\alpha).
\end{equation}
Furthermore, Antieau and Williams showed (\cite{An1}) the following
\begin{theorem}[Theorem 6, \cite{An4}]\label{same prime divisors}
Let $(X,R)$ be a locally-ringed connected topos and let $\alpha\in\operatorname{Br}(X,R)$. There exists a representative $A$ of $\alpha$ such that the prime numbers dividing $\operatorname{per}(\alpha)$ and $\operatorname{deg}(A)$ coincide.
\end{theorem}
In particular, the period and index of any $\alpha\in\operatorname{Br}(X)$ have the same prime divisors, for a connected finite CW-complex $X$. Hence, for a sufficiently large integer $e$ we have
\begin{equation}\label{ind div per}
\operatorname{ind}(\alpha)|\operatorname{per}(\alpha)^e.
\end{equation}
The topological period-index problem can be stated as follows:

\textit{For a given class $\mathscr{C}$ of finite CW complexes, find the sharp lower bound of $e$ such that (\ref{ind div per}) holds for all finite CW complex $X$ in $\textrm{C}$ and all elements $\alpha\in\textrm{Br}(X)$.}

The topological period-index problem is motivated by its analog in algebraic geometry, where the topological Brauer group of a CW complex is replaced by the usual Brauer group of a scheme, and where the period and index are defined similarly. We have the following folklore conjecture:
\begin{conjecture}[Colliot-Th{\'e}l{\`e}ne]\label{algconj}
Let $k$ be either a $C_d$-field or the function field of a $d$-dimensional variety over an algebraically closed field. Let $\alpha\in \textrm{Br}(k)$, and suppose that $\operatorname{per}(\alpha)$ is prime to the characteristic of $k$. Then
$$\operatorname{ind}(\alpha)|\operatorname{per}(\alpha)^{d-1}.$$
\end{conjecture}
Examples have been known such that $\operatorname{ind}(\alpha)=\operatorname{per}(\alpha)^{d-1}$ (\cite{Co}), so the bound is sharp if it holds. The conjecture has been proved in a few low dimensional cases, which are summarized in \cite{An1}. Very little is known in high dimensions.

There is an obvious topological analog of Conjecture \ref{algconj}, which is proposed by Antieau and Williams in \cite{An} and referred to as ``straw man'', or the topological period-index conjecture:
\begin{conjecture}[Antieau-Williams]\label{topconj}
If X is a $2d$-dimensional finite CW complex, and $\alpha\in\operatorname{Br}(X)$, then
$$\operatorname{ind}(\alpha)|\operatorname{per}(\alpha)^{d-1}.$$
\end{conjecture}
Notice that, if $X$ is a complex algebraic variety of dimension $d$, then its underlying topological space has a $2d$-dimensional cell decomposition, whence the $2d$ in the conjecture.

Antieau and Williams disproved Conjecture \ref{topconj} in \cite{An}. To state their results in consistency with this paper, we denote by $\epsilon_{p}(n)$ the greatest common divisor of $p$ and $n$. Typically $p$ will be a prime number. The notations in the following theorem is altered accordingly.
\begin{theorem}[Antieau-Williams, \cite{An}]\label{6-complex}
Let $n$ be a positive integer. There exists a connected finite CW complex $X$ of dimension $6$ equipped with a class $\alpha\in\operatorname{Br}(X)$ for which $\operatorname{per}(\alpha)=n$ and $\operatorname{ind}(\alpha)=\epsilon_{2}(n)n^2$.
\end{theorem}

Theorem \ref{6-complex} would be a special case of the following
\begin{conjecture}[Antieau-Williams, \cite{An2}]\label{awconjecture}
Let X be a finite $2d$-dimensional CW-complex, and let $\alpha\in\operatorname{Br}(X)$ have period $m=p_1^{r_1}\cdots p_k^{r_k}$. Then,
$$\operatorname{ind}(\alpha)=m^{d-1}\prod_{i=1}^{k}p_i^{v_{p_i}((d-1)!)},$$
where $v_{p_i}$ is the $p_i$-adic evaluation.
\end{conjecture}

In this paper, we show that the topological period-index conjecture fails again for $8$-dimensional CW complexes. The main result is the following
\begin{theorem}\label{main}
Let $X$ be a topological space of homotopy type of an $8$-dimensional connected finite CW-complex, and let $\alpha\in H^{3}(X;\mathbb{Z})_{\operatorname{tor}}$ be a topological Brauer class of period $n$. Then
\begin{equation}\label{bound}
\operatorname{ind}(\alpha)|\epsilon_{2}(n)\epsilon_{3}(n)n^3.
\end{equation}
In addition, if $X$ is the $8$-th skeleton of $K(\mathbb{Z}/n,2)$, and $\alpha$ is the restriction of the fundamental class $\beta_n\in H^{3}(K(\mathbb{Z}/n,2),\mathbb{Z})$, then
\begin{equation*}
\begin{cases}
\operatorname{ind}(\alpha)=\epsilon_{2}(n)\epsilon_{3}(n)n^3,\quad\textrm{$4\nmid n$,}\\
\epsilon_{3}(n)n^3|\operatorname{ind}(\alpha), \quad\textrm{$4|n$.}
\end{cases}
\end{equation*}
In particular, the sharp lower bound of $e$ such that $\operatorname{ind}(\alpha)|n^{e}$ for all $X$ and $\alpha$ is $4$.
\end{theorem}
The theorem solves the topological period-index problem for $\alpha\in\operatorname{Br}(X)$ where $X$ is an $8$-complex, and $4\nmid\operatorname{per}(\alpha)$. In particular, it implies that the topological version of the period-index conjecture fails in dimension $8$, as it does in dimension $6$. We decompose Theorem \ref{main} into two sub-theorems as follows, since the proof of the second sub-theorem requires special attention.

\begin{subtheorem}\label{main1}
Let $X$ be a topological space of homotopy type of an $8$-dimensional connected finite CW-complex, and let $\alpha\in H^{3}(X;\mathbb{Z})_{\operatorname{tor}}$ be a topological Brauer class of period $n$. Then
\begin{equation}\label{bound'}
\operatorname{ind}(\alpha)|\epsilon_{2}(n)\epsilon_{3}(n)n^3.
\end{equation}
In addition, if $X$ is the $8$-th skeleton of $K(\mathbb{Z}/n,2)$, and $\alpha$ is the restriction of the fundamental class $\beta_n\in H^{3}(K(\mathbb{Z}/n,2),\mathbb{Z})$, then
\begin{equation}\label{edit:bound}
\begin{cases}
\operatorname{ind}(\alpha)=\epsilon_{3}(n)n^3,\quad\textrm{$n$ odd,}\\
\epsilon_{3}(n)n^3|\operatorname{ind}(\alpha), \quad\textrm{$n$ even.}
\end{cases}
\end{equation}
\end{subtheorem}

The divisibility relation (\ref{bound'}) was shown by Antieau and Williams in \cite{An2}. Indeed, they proved the more general fact $\operatorname{ind}(\alpha)|m^{d-1}\prod_{i=1}^{k}p_i^{v_{p_i}((d-1)!)}$, which would be a corollary of Conjecture \ref{awconjecture} if it holds. The equation (\ref{edit:bound}) is not known before. Nonetheless, we provide a full proof of Theorem \ref{main1} for completeness.

\begin{subtheorem}\label{main2}
Let $n=2l$ for some odd integer $l$, $X$ the $8$-th skeleton of $K(\mathbb{Z}/n,2)$, and $\alpha$ the restriction of the fundamental class $\beta_n$. Then $\operatorname{ind}(\alpha)\nmid\epsilon_{3}(n)n^3$, which shows, following Theorem \ref{main1},
\[\operatorname{ind}(\alpha)=\epsilon_2(n)\epsilon_{3}(n)n^3.\]
\end{subtheorem}

The proof of (\ref{bound}) relies on twisted complex K-theory, of which details are discussed in Section 2. We prove the second paragraph of the theorem with classical obstruction theory, which is outlined as follows.

Let $m,n$ be integers. Then $\mathbb{Z}/n$ is a closed normal subgroup of $SU_{mn}$ in the sense of the following monomorphism of Lie groups:
\begin{equation*}
\mathbb{Z}/n\hookrightarrow SU_{mn}: t\mapsto e^{2\pi\sqrt{-1}t/n}\mathbf{I}_{mn},
\end{equation*}
where $\mathbf{I}_r$ is the identity matrix of degree $r$. We define the quotient group to be $P(n,mn)$. In particular, $P(n,n)$ is the projective unitary group $PU_n$, and we have the following short exact sequence of Lie groups: $$1\rightarrow\mathbb{Z}/{n}\rightarrow P(n,mn)\xrightarrow{\varphi} PU_{mn}\rightarrow 1.$$
The homotopy groups of $P(n,mn)$ in low degrees relative to $mn$ are well known:
\begin{equation}\label{homotopy grp of P(n,mn)}
\pi_{i}(P(n,mn))\cong
\begin{cases}
\mathbb{Z}/n,\quad\textrm{if $i=1$},\\
\mathbb{Z},\quad\textrm{if $1<n<2mn$, and $n$ is odd,}\\
0, \quad\textrm{if $1<n<2mn$, and $n$ is even,}\\
\mathbb{Z}/(mn)!,\quad\textrm{if $i=2mn$.}
\end{cases}
\end{equation}
This follows since $P(n,mn)$ has $SU_{mn}$ as a simply connected $n$-cover, whose homotopy groups in low dimensions follows from Bott periodicity (\cite{At2}). Consider its classifying space $\mathbf{B}P(n,mn)$, and we have a map $\mathbf{B}P(n,mn)\rightarrow K(\mathbb{Z}/n,2)$ which is the projection of $\mathbf{B}P(n,mn)$ onto the first non-trivial stage of its Postnikov tower. This map also classifies the generator of $H^{2}(\mathbf{B}P(n,mn);\mathbb{Z}/n)$.

Given a connected CW-complex $X$ such that $H^2(X;\mathbb{Z})=0$, and $\alpha\in\textrm{Br}'(X)$ of period $n$, there is a unique class $\alpha'\in H^{2}(X;\mathbb{Z}/n)$ such that $B(\alpha')=\alpha$, where $B$ is the Bockstein homomorphism. Then $\alpha'$ is classified by a map $X\rightarrow K(\mathbb{Z}/n,2)$. Therefore we have a lifting problem as shown by the following diagram:
\begin{equation}\label{lift}
\begin{tikzcd}
&\mathbf{B}P(n,mn)\arrow{d}\\
X\arrow[ur,dashed]\arrow{r}{\alpha'}&K(\mathbb{Z}/n,2)
\end{tikzcd}
\end{equation}
It can be shown, as done in later sections, that $\alpha$ is classified by a $PU_{mn}$-torsor over $X$ if and only if
the lifting problem above has a solution. If $X$ is a finite CW complex, then it suffices to study maps from $X$ into successive stages of the Posnikov tower of $\mathbf{B}P(n,mn)$, which occupies most of this paper.

In Section 2 we recapture the cohomology of Eilenberg-Mac Lane spaces necessary for our purpose; in Section 3 we introduce the twisted K-theory and the associated Atiyah-Hirzebruch spectral sequence; Sections 4, 5, and 6 are devoted to the study of the classifying spaces $\mathbf{B}P(n,mn)$, in particular their Postnikov towers, which is the technical core of this paper. In section 7 we introduce a few tricks to prove Theorem \ref{main2}

\textbf{Acknowledgement.} This paper is a revised version of the second chapter of the author's Ph.D. thesis at the University of Illinois at Chicago. The author is grateful to his co-advisor Benjamin Antieau, for proposing the question which leads to this paper and examining its earlier versions, to his advisor Brooke Shipley, for many helpful discussions, to Aldridge Bousfield, for pointing out a number of problems in earlier versions of this paper and offering solutions to them, and to Ben Williams, for a pleasant and inspiring conversation on this paper during his visit to Chicago, when he served on the author's thesis defense committee, along with all the others mentioned above. Finally the author would like to thank the referee for a very detailed review and many helpful suggestions.
\section{preliminary on the cohomology of eilenberg-mac lane spaces}
As mentioned in the introduction, the objects of interest are various stages of the Postnikov tower of the space $\mathbf{B}P(n,mn)$. It follows from (\ref{homotopy grp of P(n,mn)}) that the relevant Eilenberg-Mac Lane spaces are of the forms $K(\mathbb{Z}/n,2)$ and $K(\mathbb{Z},n)$. All the assertions made in this section are essentially consequences of \cite{Ca}.

Consider the Eilenberg-Mac Lane space $K(\mathbb{Z},n)$ for $n\geq 3$. By \cite{Ca}, the integral cohomology ring $H^{*}(K(\mathbb{Z},n);\mathbb{Z})$ in degree $\leq n+3$ is isomorphic to the following graded ring:
\begin{equation}\label{K(Z,n)}
\mathbb{Z}[\iota_n,\Gamma_n]/(2\Gamma_n),
\end{equation}
where $\iota_n$, of degree $n$, is the so-called fundamental class, and $\Gamma_n$, of degree $n+3$, is a class of order $2$. We denote by $\bar{\iota}_n, \bar{\Gamma}_n$ the mod $2$ reduction of $\iota_n$ and $\Gamma_n$ in $H^{*}(K(\mathbb{Z},n);\mathbb{Z}/2)$, respectively. Either by \cite{Ca} or by the K{\"u}nneth Theorem, there is a class $\Gamma'_n\in H^{n+2}(K(\mathbb{Z},n);\mathbb{Z}/2)$ such that $B(\Gamma'_n)=\Gamma_n$, where $B$ denotes the Bockstein homomorphism. Moreover, we consider the Steenrod square $\operatorname{Sq}^r$ and write $\overline{\operatorname{Sq}^r}$ for the following composition:
$$H^{*}(-;\mathbb{Z})\xrightarrow{\textrm{mod }2} H^{*}(-;\mathbb{Z}/2)\xrightarrow{\operatorname{Sq}^r}H^{*+r}(-;\mathbb{Z}/2),$$
where the first arrow denotes the mod $2$ reduction.
\begin{lemma}\label{Steenrod}
If $n>3$, then $\Gamma'_n=\operatorname{Sq}^{2}(\bar{\iota}_n)$. In terms of cohomology operations, this means
$$\Gamma_n=B\circ\overline{\operatorname{Sq}^2}: H^{n}(-;\mathbb{Z})\rightarrow  H^{n+3}(-;\mathbb{Z}).$$
Then the Adem relation $\operatorname{Sq}^{3}=\operatorname{Sq}^{1}\operatorname{Sq}^{2}$ implies $\overline{\Gamma}_n=\overline{\operatorname{Sq}^3}(\iota_n)$.
\end{lemma}
\begin{proof}
First we consider the case $n=3$. The path fibration $K(\mathbb{Z},2)\rightarrow *\rightarrow K(\mathbb{Z},3)$ induces a cohomological Serre spectral sequence ${^3}E_{*}^{*,*}$ with coefficients in $\mathbb{Z}/2$, such that ${^3}E_{2}^{3,0}\cong H^{3}(K(\mathbb{Z},3);\mathbb{Z}/2)\cong\mathbb{Z}/2$ is generated by $\bar{\iota}_3$; ${^3}E_{2}^{0,2}\cong H^{2}(K(\mathbb{Z},2);\mathbb{Z}/2)\cong\mathbb{Z}/2$ is generated by $\bar{\iota}_2$; ${^3}E_{2}^{5,0}\cong H^{5}(K(\mathbb{Z},3);\mathbb{Z}/2)\cong\mathbb{Z}/2$ is generated by $\Gamma'_3$; and ${^3}E_{2}^{0,4}\cong H^{4}(K(\mathbb{Z},2);\mathbb{Z}/2)\cong\mathbb{Z}/2$ is generated by $\bar{\iota}_2^2$.
The vanishing of the $E_{\infty}$ page in positive total degrees implies
\begin{equation}\label{d(iota_2)}
d_{2}(\bar{\iota}_2)=\bar{\iota}_3,
\end{equation}
and
\begin{equation}\label{d(iota_2^2)}
d_{4}(\bar{\iota}_2^2)=\Gamma'_3.
\end{equation}
Notice that $\bar{\iota}_2^2=\operatorname{Sq}^{2}(\bar{\iota}_2)$. Moreover, by Corollary 6.9 of \cite{Mc}, Steenrod squares commute with transgressions in Serre spectral sequences. The following equation then follows from (\ref{d(iota_2)}) and (\ref{d(iota_2^2)}):
\begin{equation}\label{n=3}
\Gamma'_3=d_{4}(\bar{\iota}_2^2)=d_{4}(\operatorname{Sq}^{2}(\bar{\iota}_2))=\operatorname{Sq}^{2}d_{2}((\bar{\iota}_2))=
\operatorname{Sq}^{2}(\bar{\iota}_3).
\end{equation}
This proves lemma in the case $n=3$. We verify the general case by induction on $n$. Consider the path fibration $K(\mathbb{Z},n-1)\rightarrow *\rightarrow K(\mathbb{Z},n)$. Again, by the vanishing of the $E_{\infty}$ page in positive total degrees, we have $d_{n}(\bar{\iota}_{n-1})=\bar{\iota}_n$ and  $d_{n+2}(\Gamma'_{n-1})=\Gamma'_n$.
See Figure \ref{K(Z,n)spec} for an indication of the relevant differentials. Since all the differentials in sight are transgressions, the induction is complete.
\end{proof}
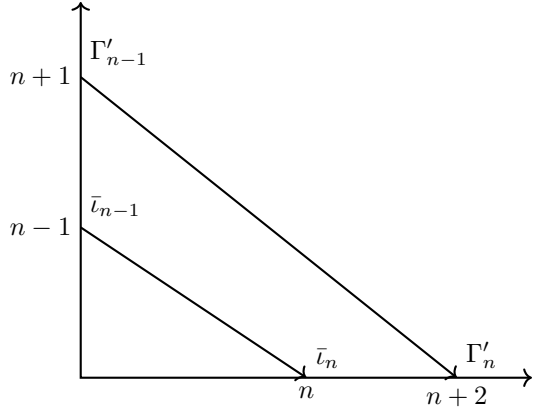
\begin{figure}[!]
\begin{tikzpicture}
\draw[step=1cm,gray, thick];
\draw [ thick, <->] (0,5)--(0,0)--(6,0);

\node [below] at (3,0) {$n$};
\node [below] at (5,0) {$n+2$};
\node [left] at (0,2) {$n-1$};
\node [left] at (0,4) {$n+1$};

\node [above right] at (0,2) {$\bar{\iota}_{n-1}$};
\node [above right] at (0,4) {$\Gamma'_{n-1}$};
\node [above right] at (3,0) {$\bar{\iota}_n$};
\node [above right] at (5,0) {$\Gamma'_{n}$};
\draw (0,2) [thick, ->] to (3,0);
\draw (0,4) [thick, ->] to (5,0);
\end{tikzpicture}\\
\caption{Low dimensional trasgressions in the mod $2$ cohomological Serre spectral sequence induced by $K(\mathbb{Z},n-1)\rightarrow * \rightarrow K(\mathbb{Z},n)$.}\label{K(Z,n)spec}
\end{figure}

We proceed to consider $K(\mathbb{Z}/n,2)$, for any positive integer $n$. By \cite{Ca}, the integral cohomology of $K(\mathbb{Z}/n,2)$ in degree $\leq 8$ is isomorphic to the following graded commutative ring:
\begin{equation}\label{K(Z/2,2)}
\mathbb{Z}[\beta_n,Q_n,R_n,\rho_n]/(n\beta_n,\epsilon_{2}(n)\beta_n^2, \epsilon_{2}(n)nQ_n,\epsilon_{3}(n)nR_n,\epsilon_3(n)\rho_n),
\end{equation}
where $\operatorname{deg}(\beta_n)=3, \operatorname{deg}(Q_n)=5, \operatorname{deg}(R_n)=7$, and $\operatorname{deg}(\rho_n)=8$. In other words, there is exactly one generator in each of the degrees $3,5,6,7$, which are, respectively, $\beta_n, Q_n, \beta_n^2, R_n$, of order
$$n,\epsilon_{2}(n)n,\epsilon_{2}(n),\epsilon_{3}(n)n,$$
and $2$ generators in degree $8$, $\beta_nQ_n$ and $\rho_n$, of order $\epsilon_2(n)$ and $\epsilon_3(n)$, respectively.
Notice that when $n$ is odd, the elements $\epsilon_{2}(n)\beta_n^2$ and $\beta_nQ_n$ are trivial. When $n$ is coprime to $3$, $\rho_n=0$.

Consider the canonical inclusion $g_{m,n}:\mathbb{Z}/n\rightarrow\mathbb{Z}/mn$, which induces a map $g_{m,n}^{(i)}:K(\mathbb{Z}/n,i)\rightarrow K(\mathbb{Z}/mn,i)$ for any integer $i>0$.

For any prime number $p$ such that $p|n$, a straight forward computation of homology of groups shows that we have the isomorphism
\begin{equation}\label{K(Z/r,1)}
 (g_{m,n}^{(1)})_*:H_2(K(\mathbb{Z}/n,1);\mathbb{Z}/p)\cong H_2(K(\mathbb{Z}/mn,1);\mathbb{Z}/p)\cong\mathbb{Z}/p.
\end{equation}
See, for example, Section 6.2 of \cite{We}. On the other hand, recall ``la transpotence''
$$\psi_p: H_{2q}(K(\mathbb{Z}/r,i);\mathbb{Z}/p)\rightarrow H_{2pq+2}(K(\mathbb{Z}/r,i+1);\mathbb{Z}/p)$$
defined in Section 6 of \cite{Ca}, which, by the example on page 6-08, \cite{Ca}, is $\mathbb{Z}/p$-linear when $p$ is odd.  We adopt the notations in Section 11 of \cite{Ca}. Let $A_r$ be the group ring of $\mathbb{Z}/r$ generated by a single element $u_r$. Then, it is an easy consequence of Section 9 of \cite{Ca} that $H_{2}(K(\mathbb{Z}/r,1);\mathbb{Z}/p)$ is generated by $\psi_p(u_r)$. It then follows from (\ref{K(Z/r,1)}) that we have
\begin{equation*}
(g_{m,n}^{(1)})_*(\psi_p(u_n))=\mu\psi_p(u_{nm}).
\end{equation*}
For some $\mu\in(\mathbb{Z}/p)^*$. When $p$ is odd, since $\psi_p$ is functorial and $\mathbb{Z}/p$-linear, we have
\begin{equation}\label{K(Z/r,2)}
(g_{m,n}^{(2)})_*((\psi_p)^2(u_n))=\mu(\psi_p)^2(u_{nm}),
\end{equation}
an equation of elements of $H_8(K(\mathbb{Z}/mn,2);\mathbb{Z}/p)$. When $p=3$, the discussion above leads to the following
\begin{lemma}\label{rho}
The induced homomorphism $H^8(g^{(2)}_{m,n})$ is a $3$-local isomorphism if $3|n$ and $0$ otherwise.
\end{lemma}
\begin{proof}
When $3\nmid n$, we have $H^8(K(\mathbb{Z}/n,2);\mathbb{Z})=0$, and there is nothing to prove. When $3|n$, by the universal coefficient theorem, it suffices to show that
$$(g_{m,n}^{(2)})_*:H_{7}(K(\mathbb{Z}/n,2);\mathbb{Z})\rightarrow H_{7}(K(\mathbb{Z}/mn,2);\mathbb{Z})$$
is an isomorphism.
It follows from Section 11 of \cite{Ca} that
$$H_{7}(K(\mathbb{Z}/r,2);\mathbb{Z})\cong\mathbb{Z}/3$$
is generated by the Bockstein of $(\psi_3)^2(u_r)$, when $3|r$. The lemma then follows from (\ref{K(Z/r,2)}).
\end{proof}

Consider the short exact sequence
$$0\rightarrow\mathbb{Z}/n\rightarrow\mathbb{Z}/mn\rightarrow\mathbb{Z}/m\rightarrow 0$$
which induces the following fiber sequence of spaces:
$$K(\mathbb{Z}/m,1)\rightarrow K(\mathbb{Z}/n,2)\rightarrow K(\mathbb{Z}/mn,2).$$
We denote the induced cohomological Serre spectral sequence in integral coefficients by $^{H}E_{*}^{*,*}$.
\begin{lemma}\label{R_n}
Let $^{H}E_{*}^{*,*}$ be as above. If $\epsilon_2(n)n|m$, then
$$H^{5}(K(\mathbb{Z}/n,2);\mathbb{Z})\cong{^{H}E}_{\infty}^{3,2}=\frac{\epsilon_2(m)}{\epsilon_2(n)}{^{H}E}_{3}^{3,2}/\operatorname{Im}
{^{H}d}_{3}^{0,4},$$
and if $\epsilon_3(n)n|m$, then
$$H^{7}(K(\mathbb{Z}/n,2);\mathbb{Z})\cong{^{H}E}_{\infty}^{3,4}=\frac{\epsilon_3(m)}{\epsilon_3(n)}{^{H}E}_{3}^{3,4}/\operatorname{Im}
{^{H}d}_{3}^{0,6}.$$
\end{lemma}
\begin{remark}
In particular, we have
\begin{equation*}
H^{5}(K(\mathbb{Z}/n,2);\mathbb{Z})\cong\mathbb{Z}/\epsilon_{2}(n)n,
\end{equation*}
and
\begin{equation*}
H^{7}(K(\mathbb{Z}/n,2);\mathbb{Z})\cong\mathbb{Z}/\epsilon_{3}(n)n.
\end{equation*}
\end{remark}
\begin{proof}
Consider the $E_2$-page
\begin{equation*}
\begin{split}
^{H}E_{2}^{s,t}\cong H^{s}(K(\mathbb{Z}/mn,2);H^t(K(\mathbb{Z}/m,1);\mathbb{Z}))\cong\\
\begin{cases}
H^{s}(K(\mathbb{Z}/mn,2);\mathbb{Z}), \textrm{ if $t=0$.}\\
H^{s}(K(\mathbb{Z}/mn,2);\mathbb{Z}/m), \textrm{ if $t>0$ and $t$ is even,}\\
0,\textrm{ if $t$ is odd.}
\end{cases}
\end{split}
\end{equation*}
This follows from the fact that, as a ring,
\begin{equation}\label{H(K(Z/m,1))}
H^{*}(K(\mathbb{Z}/m,2);\mathbb{Z})\cong\mathbb{Z}[v]/(mv),
\end{equation}
where $v$ is of degree $2$. For obvious degree reasons, we have
\begin{equation}\label{^H d_3^0,2}
^Hd_3^{0,2}: {^{H}E}_{3}^{0,2}\cong\mathbb{Z}/m\rightarrow {^{H}E}_{3}^{3,0}\cong\mathbb{Z}/mn, v\mapsto n\beta_{mn},
\end{equation}
i.e., the canonical inclusion $\mathbb{Z}/m\rightarrow\mathbb{Z}/mn$.
Since the spectral sequence is multiplicative, it follows from (\ref{^H d_3^0,2}) that
\begin{equation}\label{^H d_3^3,2}
^Hd_3^{3,2}: {^{H}E}_{3}^{3,2}\cong\mathbb{Z}/m\rightarrow {^{H}E}_{3}^{6,0}\cong\mathbb{Z}/2, v\beta_{mn}\mapsto n\beta_{mn}^2,
\end{equation}
which is surjective if $n$ is odd, and $0$ if $n$ is even. By Leibniz rule, we have
\begin{equation}\label{^H d_3^0,4}
^Hd_3^{0,4}: {^{H}E}_{3}^{0,4}\cong\mathbb{Z}/m\rightarrow {^{H}E}_{3}^{3,2}\cong\mathbb{Z}/m, v^2\mapsto 2nv\beta_{mn}.
\end{equation}
It follows from obvious degree reasons that $^HE_{\infty}^{3,2}=\operatorname{Ker}{^Hd}_{3}^{3,2}/\operatorname{Im}{^Hd}_{3}^{0,4}$. Therefore, by (\ref{^H d_3^3,2}), (\ref{^H d_3^0,4}) and $\epsilon_2(n)n|m$, we have
\begin{equation}\label{^H E^3,2}
^HE_{\infty}^{3,2}=\langle \frac{2}{\epsilon_2(n)}v\beta_{mn}\rangle/\langle 2nv\beta_{mn}\rangle=\frac{\epsilon_2(m)}{\epsilon_2(n)}{^{H}E}_{3}^{3,2}/\operatorname{Im}
{^{H}d}_{3}^{0,4}\cong\mathbb{Z}/\epsilon_2(n)n,
\end{equation}
which is isomorphic to $H^{5}(K(\mathbb{Z}/n,2);\mathbb{Z})$, and the first equation follows.

We proceed to prove the second equation in the lemma. By Leibniz rule and (\ref{^H d_3^0,2}), we have
\begin{equation*}
d_{3}^{0,6}(v^3)=3nv^2\beta_{mn}, \textrm{ and }d_3^{3,4}(v^2\beta_{mn})=2v\beta_{mn}^2=0
\end{equation*}
since $2\beta_{mn}=0$, from which it follows that
\begin{equation}\label{^H E_4^3,4}
^HE_{4}^{3,4}={^HE}_{3}^{3,4}/\operatorname{Im}d_3^{0,6}={^HE}_{3}^{3,4}/3n{^HE}_{3}^{3,4}\cong\mathbb{Z}/\epsilon_{3}(m)n.
\end{equation}
For degree reasons the only potentially nontrivial differential into or out of $^HE_{4}^{3,4}$ is
$$^Hd_{5}^{3,4}:{^HE}_{4}^{3,4}\rightarrow {^HE}_{4}^{8,0},$$
where the codomain ${^HE}_{4}^{8,0}$ is a quotient group of $H^{8}(K(\mathbb{Z}/mn,2);\mathbb{Z})$ in which $\rho_{mn}$ is nontrivial. It follows from Lemma \ref{rho} that
\begin{equation}\label{Hd_5(3,4)}
^Hd_{5}^{3,4}
\begin{cases}
 =0, \textrm{if } 3|n,\\
 \textrm{onto $\rho_{mn}$, otherwise}.
\end{cases}
\end{equation}
Hence, when $\epsilon_3(n)n|m$, we have
\begin{equation}\label{^H E^3,4}
{^HE}_{\infty}^{3,4}=\operatorname{Ker}{^Hd_{5}}^{3,4}=
\frac{\epsilon_3(m)}{\epsilon_3(n)}{^HE}_{3}^{3,4}/\operatorname{Im}d_3^{0,6}
\cong\mathbb{Z}/\epsilon_3(n)n,
\end{equation}
which is isomorphic to $H^7(K(\mathbb{Z}/n,2);\mathbb{Z})$, and the desired equation follows.
\end{proof}

The cohomology of $K(\mathbb{Z}/n,2)$ with coefficients in $\mathbb{Z}/2$ is of particular interest to us. In fact, we have a beautiful description of the cohomology ring $H^*(K(\mathbb{Z}/2, q);\mathbb{Z}/2)$ for any $q>0$. We denote the fundamental class of $H^q(K(\mathbb{Z}/2,q);\mathbb{Z}/2)$ by $b$. Recall that a finite sequence of positive integers $I=(i_1,i_2,\cdots,i_r)$ is called admissible if $i_k\geq2i_{k+1}$, for $k=1,\cdots,r-1$. The excess of $I$ is defined as
\begin{equation*}
e(I)=i_1-i_2-\cdots-i_r.
\end{equation*}
The following well-known theorem can be found, for example, in \cite{Mo}, in a slightly different form.
\begin{theorem}[Theorem 4, Chapter 9, \cite{Mo}]\label{K(Z/n,q)mod2}
When $n$ is even, the ring $H^*(K(\mathbb{Z}/n,q);\mathbb{Z}/2)$ is the polynomial ring with generators
\begin{equation*}
\operatorname{Sq}^I(b)=\operatorname{Sq}^{i_1}\operatorname{Sq}^{i_2}\cdots\operatorname{Sq}^{i_r}(b)
\end{equation*}
where $I$ runs through admissible sequences of excess $e(I)<q$, with the exception, in the case $4|n$, and $i_r=1$, $\operatorname{Sq}^I(b)$ is replaced by
\begin{equation*}
\operatorname{Sq}^{i_1}\operatorname{Sq}^{i_2}\cdots\operatorname{Sq}^{i_{r-1}}(b'),
\end{equation*}
where $b'$ is the mod $2$ reduction of the generator of $H^{q+1}(K(\mathbb{Z}/n,q);\mathbb{Z})$, and $\operatorname{Sq}^1(b)=0$.
\end{theorem}

In the special case of $q=2$, we have
\begin{corollary}\label{K(Z/n,2)mod 2}
When $n$ is even, we have the isomorphism
\begin{equation*}
H^*(K(\mathbb{Z}/n,2);\mathbb{Z}/2)=\cong\mathbb{Z}/2[b_2, b_3, b_5]
\end{equation*}
where $b_2=b$ is the fundamental class, $\operatorname{Sq}^1b_2=0$ when $4|n$ and $\operatorname{Sq}^1b_2=b_3$ otherwise, and $b_5=\operatorname{Sq}^2b_3$.
\end{corollary}
We conclude this section with the following.
\begin{proposition}\label{R_n mod 2}
The mod $2$ reduction of $R_2\in H^7(K(\mathbb{Z}/2,2);\mathbb{Z})$ is $b_2^{2}b_3\in H^7(K(\mathbb{Z}/2,2);\mathbb{Z}/2)$. In particular, it is nontrivial.
\end{proposition}
\begin{proof}
We know that $\operatorname{Sq}^1$ is the composition of the Bockstein homomorphism followed by the mod $2$ reduction. (\cite{Mo}, for example.) It follows from (\ref{K(Z/2,2)}) that $H^8(K(\mathbb{Z}/2,2);\mathbb{Z})$ is a $2$-torsion group, from which it follows that the mod $2$ reduction
$$H^8(K(\mathbb{Z}/2,2);\mathbb{Z})\rightarrow H^8(K(\mathbb{Z}/2,2);\mathbb{Z}/2)$$
is injective. Furthermore, $H^8(K(\mathbb{Z}/2,2);\mathbb{Z})$ is generated by a single element $R_2$. Therefore, it suffices to show $\operatorname{Sq}^1(b_2^{2}b_3)=0$, which implies the Bockstein homomorphism sends $b_2^{2}b_3$ to $0$, i.e., $b_2^{2}b_3$ is the mod $2$ reduction of some nonzero integral class, which may only be $R_2$. Indeed, by Cartan's formula, we have
\begin{equation*}
\begin{split}
\operatorname{Sq}^1(b_2^{2}b_3)
=&\operatorname{Sq}^1(b_2)b_2b_3+b_2\operatorname{Sq}^1(b_2)b_3+b_2^2\operatorname{Sq}^1(b_3)\\
=&2\operatorname{Sq}^1(b_2)b_2b_3+b_2^2\operatorname{Sq}^1\operatorname{Sq}^1(b_2)\\
=&0,
\end{split}
\end{equation*}
where the last equation follows from the Adem relation $\operatorname{Sq}^1\operatorname{Sq}^1=0$.
\end{proof}

\section{twisted K-theory and the atiyah-hirzebruch spectral sequence}
For a connected topological space $X$ and a class $\alpha\in\operatorname{Br}'(X)=H^{3}(X;\mathbb{Z})_{\textrm{Tor}}$, Donovan-Karoubi (\cite{Do}) and Atiyah-Segal (\cite{At})defined  the twisted complex K-theory of $X$ with respect to $\alpha$, which we denote by $KU(X)_{\alpha}$, following the convention in \cite{An1}.

Similar to the usual, untwisted complex K-Theory, there is a twisted version of the Atiyah-Hirzebruch spectral sequence, $\tilde{E}_{*}^{*,*}$, such that
\begin{equation*}
\tilde{E}_{2}^{s,t}\cong
\begin{cases}
H^{s}(X;\mathbb{Z}),\quad\textrm{if $t$ is even,}\\
0, \quad\textrm{if $t$ is odd,}
\end{cases}
\end{equation*}
and converges to $KU(X)_{\alpha}$ when $X$ is a finite CW complex. For more details, see \cite{At} and \cite{At1}. The spectral sequence is closely related to the index of $\alpha$, as shown in the following
\begin{theorem}\label{AH diff}
Let $X$ be a connected finite CW complex and let $\alpha\in\operatorname{Br}(X)$. Consider $\tilde{E}_{*}^{*,*}$, the twisted Atiyah-Hirzebruch spectral sequence with respect to $\alpha$ with differentials $\tilde{d}_{r}^{s,t}$ with bi-degree $(r, -r+1)$. Then we have $\tilde{E}_{2}^{0,0}\cong\mathbb{Z}$, and any $\tilde{E}_{r}^{0,0}$ with $r>2$ is a subgroup of $\mathbb{Z}$ and therefore generated by a positive integer. The subgroup $\tilde{E}_{3}^{0,0}$ (resp.  $\tilde{E}_{\infty}^{0,0}$) is generated by $\operatorname{per}(\alpha)$ (resp. $\operatorname{ind}(\alpha))$.
\end{theorem}
Moreover, we have a rank map $KU^{0}(X)_{\alpha}\rightarrow\mathbb{Z}$ (See Section 2.5 of \cite{An1}) of which the image is generated by $\operatorname{ind}(\alpha)$. Theorem \ref{AH diff} is an immediate consequence of Proposition 2.21 and Lemma 2.23 of \cite{An1}. It has the following consequence:
\begin{corollary}\label{upper bound}
Let $X$ be a connected $8$-dimensional CW-complex, and let $\alpha\in\textrm{Br}'(X)=H^{3}(X;\mathbb{Z})_{\textrm{tor}}$ be such that $\operatorname{per}(\alpha)=n$. Then $\operatorname{ind}(\alpha)|\epsilon_{2}(n)\epsilon_{3}(n)n^3$.
\end{corollary}
\begin{proof}
First we fix a CW-complex structure on the Eilenberg-Mac Lane space $K(\mathbb{Z}/n,2)$ and take $X$ to be $\operatorname{sk}_8(K(\mathbb{Z}/n,2))$, the $8$th skeleton of $K(\mathbb{Z}/n,2)$. Then the corresponding twisted Atiyah-Hirzebruch spectral sequence is shown in Figure \ref{twisted AH}, where one readily sees that the only differentials out of $\tilde{E}_{*}^{0,0}$ with non-trivial codomains are $\tilde{d}_{3}^{0,0}$, $\tilde{d}_{5}^{0,0}$ and $\tilde{d}_{7}^{0,0}$, whose codomains are, respectively, subquotients of $\tilde{E}_{2}^{3,-2}\cong H^{3}(K(\mathbb{Z}/n);\mathbb{Z})$, $\tilde{E}_{2}^{5,-4}\cong H^{5}(K(\mathbb{Z}/n);\mathbb{Z})$ and $\tilde{E}_{2}^{7,-6}\cong H^{7}(K(\mathbb{Z}/n);\mathbb{Z})$. As discussed in Section 1, the three groups above are all cyclic, of order $n$, $\epsilon_{2}(n)n$, and $\epsilon_{3}(n)n$ respectively, from which the desired result follows for $X=\operatorname{sk}_8(K(\mathbb{Z}/n,2))$.

For a general $X$ and $\alpha$, choose $\alpha'\in H^{2}(X;\mathbb{Z}/n)$ such that $B(\alpha')=\alpha$, where $B$ is the Bockstein homomorphism. then $\alpha'$ is classified by a cell map $f: X\rightarrow K(\mathbb{Z}/n,2)$ such that $f^{*}(\beta_n)=\alpha$, where $\beta_n$ is the canonical generator of $H^{2}(K(\mathbb{Z}/n,2);\mathbb{Z})$ as discussed in Section 2. The corollary then follows from the functoriality of the twisted Atiyah-Hirzebruch spectral sequence. The idea of the proof is indicated in Figure \ref{twisted AH}.
\end{proof}

\begin{figure}[!]
\begin{tikzpicture}
\draw[step=1cm,gray, thick] (0,0) grid (7,-6);
\draw [ thick, <->] (0,-6)--(0,0)--(8,0);

\node [above] at (3,0) {$3$};
\node [above] at (5,0) {$5$};
\node [above] at (7,0) {$7$};
\node [left] at (0,-2) {$-2$};
\node [left] at (0,-4) {$-4$};
\node [left] at (0,-6) {$-6$};

\node [below right] at (0,0) {$\mathbb{Z}$};
\node [below right] at (3,0) {$\mathbb{Z}/n$};
\node [below right] at (5,0) {$\mathbb{Z}/\epsilon_{2}(n)n$};
\node [below right] at (7,0) {$\mathbb{Z}/\epsilon_{3}(n)n$};
\draw (0,0) [thick, ->] to (3,-2);
\draw (0,0) [thick, ->] to (5,-4);
\draw (0,0) [thick, ->] to (7,-6);
\end{tikzpicture}
\caption{The twisted Atiyah-Hirzebruch spectral sequence associated to the $8$th skeleton of $K(\mathbb{Z}/n,2)$.}\label{twisted AH}
\end{figure}
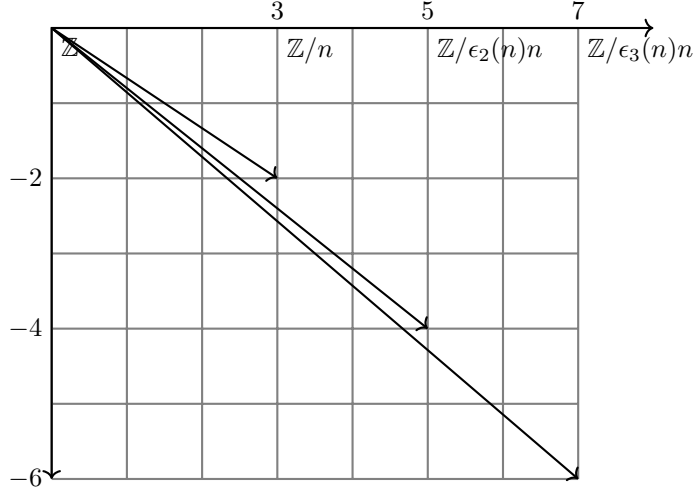

\section{the space $\mathbf{B}P(n,mn)$ and its low dimensional postnikov decomposition}
Let $m,n$ be integers. Recall that in Section 1 we defined a Lie group $P(n,mn)$ which fits in the following exact sequence:
$$1\rightarrow\mathbb{Z}/{m}\rightarrow P(n,mn)\xrightarrow{\varphi} PU_{mn}\rightarrow 1.$$
Applying the classifying space functor, we obtain a fiber sequence
\begin{equation}\label{fiberseq}
\mathbf{B}\mathbb{Z}/{m}\rightarrow\mathbf{B}P(n,mn)\xrightarrow{\mathbf{B}\varphi} \mathbf{B}PU_{mn}.
\end{equation}
The space $\mathbf{B}P(n,mn)$ plays an important role in the study of topological period-index problem.
As we mentioned in the introduction, $\pi_{1}(P(n,mn))\cong\mathbb{Z}/n$, and consequently $\mathbf{B}P(n,mn)$ is a simply connected space with $\pi_{2}(\mathbf{B}P(n,mn))\cong\mathbb{Z}/n$. Therefore we have a projection onto the $2$nd stage of Postnikov Tower $\mathbf{B}P(n,mn)\rightarrow K(\mathbb{Z}/n,2)$. In particular, $P(n,n)=PU_n$, and we have the following commutative diagram
\begin{equation}\label{Postnikov}
\begin{tikzcd}
\mathbf{B}P(n,mn)\arrow{d}\arrow{r}&\mathbf{B}PU_{mn}\arrow{d}\\
K(\mathbb{Z}/n,2)\arrow{r}&K(\mathbb{Z}/mn,2)
\end{tikzcd}
\end{equation}
where the vertical arrows are the projections to the $2$nd stages of the respective Postnikov towers.
\begin{proposition}\label{Z/n to Z/mn}
The bottom arrow in the diagram (\ref{Postnikov}) is induced by the canonical inclusion $\mathbb{Z}/n\hookrightarrow\mathbb{Z}/mn$.
\end{proposition}
In fact, this follows from the fact that $\mathbf{B}\psi$ induces a morphism on the $2$nd homotopy groups which is the inclusion described above.

Delooping the first term of the fiber sequence (\ref{fiberseq}), we obtain another fiber sequence
\begin{equation}\label{fiberseq'}
\mathbf{B}P(n,mn)\xrightarrow{\mathbf{B}\varphi}\mathbf{B}PU_{mn}\xrightarrow{\eta} K(\mathbb{Z}/m,2),
\end{equation}
which leads to the following
\begin{proposition}\label{eta}
The second arrow $\eta$ in the fiber sequence (\ref{fiberseq'}) fits in the following commutative diagram:
\begin{equation*}
\begin{tikzcd}
\mathbf{B}P(n,mn)\arrow{d}\arrow{r}&\mathbf{B}PU_{mn}\arrow{d}\arrow{dr}{\eta}\\
K(\mathbb{Z}/n,2)\arrow{r}&K(\mathbb{Z}/mn,2)\arrow{r}&K(\mathbb{Z}/m,2)
\end{tikzcd}
\end{equation*}
in which the square is diagram (\ref{Postnikov}) and the bottom row is a fiber sequence induced by delooping twice the canonical short exact sequence
$$0\rightarrow\mathbb{Z}/n\rightarrow\mathbb{Z}/mn\rightarrow\mathbb{Z}/m\rightarrow 0.$$
\end{proposition}

Let $X$ be a finite CW complex and let $\alpha\in\textrm{Br}(X)=H^{3}(X;\mathbb{Z})_{\textrm{tor}}$ be of order $n$. Recall the lifting problem (\ref{lift}) discussed in the introduction, as shown by the following diagram.
\begin{equation}\label{lift'}
\begin{tikzcd}
&\mathbf{B}P(n,mn)\arrow{d}\\
X\arrow[ur,dashed]\arrow{r}{\alpha'}&K(\mathbb{Z}/n,2)
\end{tikzcd}
\end{equation}
We have the following
\begin{proposition}\label{lift to BP}
Let $X$, $\alpha$ be as above. Furthermore, suppose that $H^2(X;\mathbb{Z})=0$. Then $\alpha$ is classified by an Azumaya algebra of degree $mn$ if and only if the lift in diagram (\ref{lift'}) exists.
\end{proposition}
\begin{proof}
The ``if'' part follows easily by post-composing a lift $X\rightarrow\mathbf{B}P(n,mn)$ with $\mathbf{B}\varphi$ (See (\ref{fiberseq})). To prove the ``only if'' part, suppose that $\alpha$ is classified by a map $f:X\rightarrow\mathbf{B}PU_{mn}$. Since $H^2(X;\mathbb{Z})=0$, there is a unique $\alpha'\in H^2(X;\mathbb{Z}/n)$ such that $B(\alpha')=\alpha$, where $B$ is the Bockstein homomorphism. Moreover, let $\alpha''$ be the image of $\alpha'$ under the canonical map $K(\mathbb{Z}/n,2)\rightarrow K(\mathbb{Z}/mn,2)$, then $\alpha''$ is the unique class in $H^2(X;\mathbb{Z}/mn)$ such that $B(\alpha'')=\alpha$. The uniqueness of $\alpha''$ indicates that the map $f$ above fits in the following commutative diagram
\begin{equation}\label{classify alpha}
\begin{tikzcd}
&\mathbf{B}P(n,mn)\arrow{r}\arrow{d}&\mathbf{B}PU_{mn}\arrow{r}{\eta}\arrow{d}&K(\mathbb{Z}/m,2)\arrow{d}{=}\\
X\arrow[ur,dashed]\arrow{urr}[swap,pos=0.7]{f}
\arrow{r}[swap]{\alpha'}&K(\mathbb{Z}/n,2)\arrow{r}&K(\mathbb{Z}/mn,2)\arrow{r}&K(\mathbb{Z}/m,2)
\end{tikzcd}
\end{equation}
where the square in the middle is the one in Proposition \ref{Postnikov}, and both the top and bottom rows of the $3$ by $2$ rectangular diagram are fiber sequences. The bottom row being a contractible map, a simple diagram chasing shows that the lift indicated by the dashed arrow exists.
\end{proof}

We denote integral cohomological Serre spectral sequence associated to (\ref{fiberseq}) by $(E_{*}^{*,*}, d_{*}^{*,*})$, of which the $E_2$ page is
\begin{equation}\label{E_2}
E_{2}^{s,t}\cong H^{s}(\mathbf{B}PU_{mn}; H^{t}(\mathbf{B}\mathbb{Z}/m))\cong
\begin{cases}
H^{s}(\mathbf{B}PU_{mn}; \mathbb{Z}), \quad\textrm{if }t=0;\\
H^{s}(\mathbf{B}PU_{mn}; \mathbb{Z}/m),\quad \textrm{if $t>0$ is even;}\\
0, \quad\textrm{if $t$ is odd.}
\end{cases}
\end{equation}

This follows from the fact that, as a ring,
\begin{equation}\label{H(K(Z/m,1))}
H^{*}(\mathbf{B}\mathbb{Z}/m)\cong\mathbb{Z}[v]/(mv),
\end{equation}
 where $v$ is of degree $2$. As for the cohomology of $\mathbf{B}PU_{mn}$, we have the following
\begin{theorem}[\cite{Gu}, Theorem 1.1]\label{Cohomology of BPU}
 For an integer $n>1$,  $H^{*}(\mathbf{B}PU_n;\mathbb{Z})$ in degrees $\leq 10$ is isomorphic to the following graded ring:
 $$\mathbb{Z}[e_2, \cdots, e_{j_n}, x_1, y_{3,0}, y_{2,1}]/I_n.$$
 Here $e_i$ is of degree  $2i$, $j_n=\textrm{min}\{5, n\}$. The degrees of $x_1, y_{3,0}, y_{2,1}$ are $3, 8, 10$, respectively. The ideal $I_n$ is generated by
 \[nx_1,\quad \epsilon_2(n)x_1^2,\quad \epsilon_3(n)y_{3,0},\quad \epsilon_2(n)y_{2,1},\]
 \[\delta(n)e_{2}x_1,\quad (\delta(n)-1)(y_{2,1}-e_{2}x_1^2),\quad e_{3}x_1,\]
 where $\epsilon_p(n)=\operatorname{gcd}(p,n)$, and
 \begin{equation*}
 \delta(n)=
 \begin{cases}
  2, \textrm{ if }n=4l+2\textrm{ for some integer }l,\\
  1, \textrm{ otherwise}.
 \end{cases}
 \end{equation*}
\end{theorem}
The degreewise cohomology groups with coefficients in an arbitrary ring follow immediately from the theorem above, together with the K{\"u}nneth theorem. We will simply refer to Theorem \ref{Cohomology of BPU} for them.

Consider the quotient map $SU_{mn}\rightarrow P(n,mn)$, which is a simply connected cover with Deck transformation group $\mathbb{Z}/n$. Therefore, we have
\begin{equation}\label{H^3}
\begin{cases}
H^{1}(\mathbf{B}P(n,mn;\mathbb{Z}))\cong H^{2}(\mathbf{B}P(n,mn);\mathbb{Z})=0,\\
H^{3}(\mathbf{B}P(n,mn);\mathbb{Z})\cong H_{2}(\mathbf{B}P(n,mn);\mathbb{Z})\cong\pi_{2}(\mathbf{B}P(n,mn))\cong
\pi_{1}(P(n,mn))\cong\mathbb{Z}/n,
\end{cases}
\end{equation}
which leads to the following

\begin{lemma}\label{d_3}
In the Serre spectral sequence $(E_{*}^{*,*}, d_{*}^{*,*})$ associated to the fiber sequence
$$\mathbf{B}\mathbb{Z}/{m}\rightarrow\mathbf{B}P(n,mn)\xrightarrow{\mathbf{B}\varphi} \mathbf{B}PU_{mn},$$
the differential $d_{3}^{0,2}$ is a monomorphism. By choosing the generator of $E_{2}^{0,2}\cong\mathbb{Z}/m$ correctly, $d_{3}^{0,2}$ can be taken as the canonical inclusion $\mathbb{Z}/m\hookrightarrow\mathbb{Z}/mn$. In particular, $H^{3}(\mathbf{B}P(n,mn);\mathbb{Z})$ is generated by $x_1'$, the image of $x_1$ under the homomorphism $H^{3}(\mathbf{B}PU_{mn};\mathbb{Z})\rightarrow H^{3}(\mathbf{B}P(n,mn);\mathbb{Z})$ induced by the quotient map.
\end{lemma}

Lemma \ref{d_3} has the following consequence:
\begin{proposition}\label{H^5(BP)}
Let $m,n$ be positive integers. Then $\epsilon_{2}(n)n|m$ if and only if
\begin{equation*}
H^{5}(\mathbf{B}P(n,mn);\mathbb{Z})\cong\mathbb{Z}/\epsilon_{2}(n)n.
\end{equation*}
\end{proposition}
\begin{proof}
See Figure \ref{spec seq E} for the spectral sequence discussed here. Notice $E_2^{5,0}=0$ from Theorem \ref{Cohomology of BPU}. Then for obvious degree reasons the only nontrivial entry of the $E_2$-page of total degree $5$ is $E_{2}^{3,2}\cong\mathbb{Z}/m$, from which it follows that
\begin{equation*}
H^5(\mathbf{B}P(n,mn);\mathbb{Z})\cong E_{\infty}^{3,2}.
\end{equation*}
The proposition then follows from the same computation as in the proof of the first statement of Lemma \ref{R_n}, only with $\beta_{mn}$ replaced by $x_1$.
\end{proof}
\begin{remark}
Indeed, the generator of $H^2(\mathbf{B}P(n,mn);\mathbb{Z}/2)$
\begin{equation*}
\mathbf{B}P(n,mn)\rightarrow K(\mathbb{Z}/n,2)
\end{equation*}
induces a homomorphism taking the generator $Q_n$ of $H^5(K(\mathbb{Z}/n,2);\mathbb{Z})$ to the generator of $H^5(\mathbf{B}P(n,mn);\mathbb{Z})$ when $\epsilon_2(n)n|m$.
\end{remark}
\begin{figure}
\begin{tikzpicture}
\draw[step=1cm,gray, thick] (0,0) grid (7,7);
\draw [ thick, <->] (0,7)--(0,0)--(7,0);

\node [below] at (2,0) {$2$};
\node [below] at (3,0) {$3$};
\node [below] at (4,0) {$4$};
\node [below] at (5,0) {$5$};
\node [below] at (6,0) {$6$};
\node [left] at (0,2) {$2$};
\node [left] at (0,4) {$4$};
\node [left] at (0,6) {$6$};
\filldraw [gray] (2,2) circle (1pt);
\filldraw [gray] (2,4) circle (1pt);
\filldraw [gray] (5,2) circle (1pt);

\node [above right] at (0,2) {$\mathbb{Z}/m$};
\node [above right] at (0,4) {$\mathbb{Z}/m$};
\node [above right] at (3,0) {$\mathbb{Z}/mn$};
\node [above right] at (4,0) {$\mathbb{Z}$};
\node [below right] at (2,2) {$\mathbb{Z}/m$};
\node [below right] at (2,4) {$\mathbb{Z}/m$};
\node [above right] at (5,0) {$0$};
\node [above right] at (6,0) {$\mathbb{Z}\oplus\mathbb{Z}/2$};
\node [above right] at (5,2) {$\mathbb{Z}/2$};
\draw (0,2) [thick, ->] to node [above] {$\times n$} (3,0);
\draw (0,4) [thick, ->] to node [above] {$\times 2n$}(3,2);
\draw (0,6) [thick, ->] to node [above] {$\times 3n$}(3,4);
\draw (3,2) [thick, ->] to node [above] {$\times n$} (6,0);
\end{tikzpicture}\\
\caption{The $E_3$-page of the spectral sequence $E_{*}^{*,*}$.}\label{spec seq E}
\end{figure}
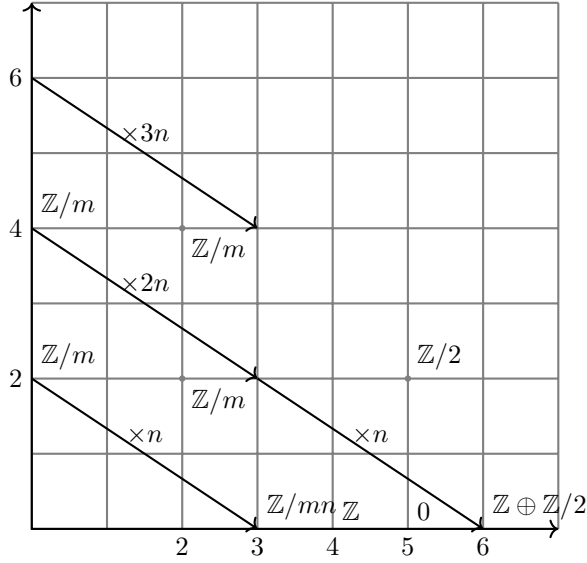

In the proof of Proposition \ref{H^5(BP)}, we observe that when $n$ is odd and $m$ is even, the differentials $d_3^{0,2}$ and $d_3^{3,2}$ annihilate the $2$-torsion elements in $E_{3}^{*,0}\cong H^{*}(\mathbf{B}P(n,mn);\mathbb{Z})$. This is a special case of a more general argument. By the definition of $P(n,mn)$, we have the following fiber sequence
$$\mathbf{B}SU_{mn}\rightarrow\mathbf{B}P(n,mn)\rightarrow K(\mathbb{Z}/n,2).$$
We consider the associated cohomological Serre spectral sequence with integral coefficients, of which the $E_2$-page has no $p$-torsion for any prime $p$ not dividing $n$, from which we deduce
\begin{lemma}\label{p-torsion}
For a prime $p$, $H^*(\mathbf{B}P(n,mn);\mathbb{Z})$ has no nontrivial $p$-torsion if $p\nmid n$.
\end{lemma}
Lemma \ref{p-torsion} has the following immediate
\begin{corollary}\label{p-tor in E}
Let $p$ be a prime such that $p|m$ and $p\nmid n$. Then all $p$-torsion element of $E_{2}^{*,0}$ vanish in the $E_{\infty}$-page.
\end{corollary}

Recall from Theorem \ref{Cohomology of BPU} that the torsion subgroup of $E_2^{8,0}\cong H^{8}(\mathbf{B}PU_r;\mathbb{Z})$ is $\mathbb{Z}/3$ if $3|r$ and $0$ otherwise.
\begin{corollary}\label{d_5}
If $3|m$ and $3\nmid n$, then the differential
$$d_5^{3,4}: E_5^{3,4}\rightarrow E_{5}^{8,0}=E_2^{8,0}$$
is a surjection onto its subgroup $\mathbb{Z}/3$.
\end{corollary}
\begin{proof}
It follows for degree reasons that $d_5^{3,4}$ is the only possibly nontrivial differential towards $E_{*}^{8,0}$, and in particular, it follows that $E_{5}^{8,0}=E_2^{8,0}$. The fact that $d_5^{3,4}$ is onto $\mathbb{Z}/3$ follows from Corollary \ref{p-tor in E}.
\end{proof}
We proceed to consider the Postnikov tower of $\mathbf{B}P(n,mn)$. Recall the low-dimensional homotopy groups of $\mathbf{B}P(n,mn)$:
\begin{equation}
\pi_{i}(\mathbf{B}P(n,mn))\cong
\begin{cases}
\mathbb{Z}/n, \quad i=2,\\
\mathbb{Z}, \quad 2<i<2mn, \quad i\textrm{ even},\\
0, \quad 0<i<2mn, \quad i\textrm{ odd}.\\
\end{cases}
\end{equation}
We denote the $i$th stage of the Postnikov tower of a simply connected topological space $X$ by $X[i]$, and the $i$th $k$-invariant by $\kappa_i$. Then we have part of the Postnikov system of $\mathbf{B}P(n,mn)$ as follows:
\begin{equation}
\begin{tikzcd}
K(\mathbb{Z},4)\arrow{r}&\mathbf{B}P(n,mn)[4]=\mathbf{B}P(n,mn)[5]\arrow{d}&\\
&\mathbf{B}P(n,mn)[3]=K(\mathbb{Z}/n,2)\arrow{r}{\kappa_3}& K(\mathbb{Z},5)
\end{tikzcd}
\end{equation}
In general we have $\mathbf{B}P(n,mn)[2i]=\mathbf{B}P(n,mn)[2i+1]$ for all $n>0$ even and $i<n$, since in such cases we have $\pi_{2i+1}(\mathbf{B}P(n,mn))=0$. By (\ref{K(Z/2,2)}) and Proposition \ref{H^5(BP)}, we have
\begin{equation*}
H^{5}(\mathbf{B}P(n,mn);\mathbb{Z})\cong H^{5}(K(\mathbb{Z}/n,2);\mathbb{Z})\cong\mathbb{Z}/\epsilon_{2}(n)n\cong H^{5}(K(\mathbb{Z}/n,2)\times K(\mathbb{Z},4);\mathbb{Z})
\end{equation*}
if and only if $\epsilon_2(n)n|m$, which implies the following
\begin{proposition}\label{kappa3}
Let $m,n$ be positive integers. Then $\epsilon_{2}(n)n|m$ if and only if in the Postnikov tower of $\mathbf{B}P(n,mn)$, we have $\kappa_3=0$. or equivalently, we have
$$\mathbf{B}P(n,mn)[5]=\mathbf{B}P(n,mn)[4]\simeq K(\mathbb{Z}/n,2)\times K(\mathbb{Z},4).$$
\end{proposition}
\begin{remark}
This is essentially the main result of \cite{An}.
\end{remark}
The integral cohomology groups of $\mathbf{B}P(n,mn)$ in degree $\leq 5$ are immediate from the proposition above. In particular, we have
\begin{corollary}\label{H^4(BP)}
As in Proposition \ref{kappa3}, we assume that $\epsilon_{2}(n)n|m$.
\begin{enumerate}
\item $H^{4}(\mathbf{B}P(n,mn);\mathbb{Z})\cong\mathbb{Z}$. We denote its generator by $e'_2$.
\item Recall the map $\mathbf{B}\varphi: \mathbf{B}P(n,mn)\rightarrow \mathbf{B}PU_{mn}$ induced by the quotient map $\varphi$. The induced homomorphism
    $$(\mathbf{B}\varphi)^{*}: H^{4}(\mathbf{B}PU_{mn};\mathbb{Z})\cong\mathbb{Z}\rightarrow H^{4}(\mathbf{B}P(n,mn);\mathbb{Z})\cong\mathbb{Z}$$
    is the multiplication by $\epsilon_{2}(n)mn$.
\end{enumerate}
\end{corollary}
\begin{proof}
The statement (1) follows immediately from Proposition \ref{kappa3}. To prove (2), consider the spectral sequence $E_{*}^{*,*}$ as in (\ref{E_2}). Notice $E_{2}^{5,0}\cong H^{5}(\mathbf{B}PU_{mn};\mathbb{Z})=0$, from which it follows that
\begin{equation}\label{E_3^2,2}
E_{\infty}^{2,2}=E_{2}^{2,2}\cong H^{2}(\mathbf{B}PU_{mn};\mathbb{Z}/m)\cong\mathbb{Z}/m.
\end{equation}
For the same reason we have $E_{\infty}^{0,4}\cong\textrm{Ker}d_{3}^{0,4}$. By the Leibniz rule we have $d_{3}(v^2)=2vd_{3}(v)=2nvx_1$, which implies that $E_{\infty}^{0,4}$ is the subgroup of $E_{2}^{0,4}$ of $\epsilon_{2}(n)n$-torsion elements, i.e.,
\begin{equation}\label{E_3^0,4}
E_{\infty}^{0,4}\cong\textrm{Ker}d_{3}^{0,4}=E_{4}^{0,4}\cong\mathbb{Z}/\epsilon_{2}(n)n.
\end{equation}
The equations (\ref{E_3^2,2}) and (\ref{E_3^0,4}), together with (1) of the corollary imply (2).
\end{proof}

We proceed to make a similar assertion on $H^{6}(\mathbf{B}P(n,mn);\mathbb{Z})$. To do so we need the following
\begin{lemma}\label{noncyclic}
When $\epsilon_{2}(n)n|m$, the abelian group $H^{7}(\mathbf{B}P(n,mn);\mathbb{Z})$ is isomorphic to  $\mathbb{Z}/\epsilon_{3}(n)n\times\mathbb{Z}/n\times\mathbb{Z}/2$ modulo a cyclic subgroup. In particular, $H^{7}(\mathbf{B}P(n,mn);\mathbb{Z})$ is \textbf{not} a cyclic group when $n$ is even.
\end{lemma}
\begin{proof}
It follows from (\ref{kappa3}) that
$$H^{7}(\mathbf{B}P(n,mn)[5];\mathbb{Z})\cong H^{7}(K(\mathbb{Z}/n,2)\times K(\mathbb{Z},4);\mathbb{Z})\cong\mathbb{Z}/\epsilon_{3}(n)n\times\mathbb{Z}/n\times\mathbb{Z}/2.$$
Therefore
$$H^{7}(\mathbf{B}P(n,mn);\mathbb{Z})\cong H^{7}(\mathbf{B}P(n,mn)[6];\mathbb{Z})\cong\mathbb{Z}/\epsilon_{3}(n)n\times\mathbb{Z}/n\times\mathbb{Z}/2/(\kappa_5),$$
and the result follows.
\end{proof}

\begin{corollary}\label{H^6(BP)}
Suppose $\epsilon_{2}(n)n|m$.
\begin{enumerate}
\item We have
\begin{equation*}
H^{6}(\mathbf{B}P(n,mn);\mathbb{Z})\cong
\begin{cases}
\mathbb{Z}\oplus\mathbb{Z}/2, \textrm{ if $n$ is even},\\
\mathbb{Z}, \textrm{ if $n$ is odd.}
\end{cases}
\end{equation*}

\item When $n$ is even, the subgroup $\mathbb{Z}/2$ of $H^{6}(\mathbf{B}P(n,mn);\mathbb{Z})$ is generated by $(x_1')^2$, where $x_1'=\mathbf{B}\varphi^*(x_1)$. Furthermore, $\mathbf{B}\varphi$ induces a homomorphism $$H^{6}(\mathbf{B}PU_{mn};\mathbb{Z})/(x_1^2)\cong\mathbb{Z}\rightarrow H^{6}(\mathbf{B}P(n,mn);\mathbb{Z})/(\mathbf{B}\varphi(x_1^2))\cong\mathbb{Z}$$
    which is the multiplication by
\begin{equation*}
\begin{cases}
\epsilon_{3}(\frac{m}{n})nm^{2}, \textrm{if $n$ is even, or $m,n$ are both odd,}\\
\epsilon_{3}(\frac{m}{n})nm^{2}/2, \textrm{if $n$ is odd, and $4|m$,}\\
\epsilon_{3}(\frac{m}{n})nm^{2}/4, \textrm{if $n$ is odd, and $m=2(2l+1)$ for some integer $l$.}
\end{cases}
\end{equation*}
\end{enumerate}
\end{corollary}
\begin{proof}
Consider the $6$th stage of the Postnikov tower of $\mathbf{B}P(n,mn)$ as described by the following diagram:
\begin{equation*}
\begin{tikzcd}
K(\mathbb{Z},6)\arrow{r}&\mathbf{B}P(n,mn)[6]\arrow{d}&\\
&\mathbf{B}P(n,mn)[5]=K(\mathbb{Z}/n,2)\times K(\mathbb{Z},4)\arrow{r}{\kappa_5}& K(\mathbb{Z},7)
\end{tikzcd}
\end{equation*}
from which it follows that
\begin{equation*}
H^{6}(\mathbf{B}P(n,mn);\mathbb{Z})\cong\mathbb{Z}\oplus H^{6}(K(\mathbb{Z}/n,2)\times K(\mathbb{Z},4);\mathbb{Z})\cong
\begin{cases}
\mathbb{Z}\oplus\mathbb{Z}/2, \textrm{ if $n$ is even},\\
\mathbb{Z}, \textrm{ if $n$ is odd}.
\end{cases}
\end{equation*}
from which (1) follows.

Consider the spectral sequence $E_{*}^{*,*}$. When $n$ is even, we have
$$E_{2}^{6,0}\cong H^{6}(\mathbf{B}PU_{mn};\mathbb{Z})\cong\mathbb{Z}\oplus\mathbb{Z}/2,$$
of which the $2$-torsion subgroup is generated by $x_1^2$. By Lemma \ref{d_3}, the image of $d_{3}^{3,2}$ is generated by $nx_1^2=0$, since $n$ is even. Therefore we have
\begin{equation}\label{d_3^(4,2)}
d_{3}^{3,2}=0.
\end{equation}
For obvious degree reasons there is no other nontrivial differential hitting the entry $(6,0)$. Hence the first half of (2) follows.

To prove the second half of (2), it suffices to show that $E_{\infty}^{s,t}$ such that $s+t=6, t>0$ are all finite, of which the product of the cardinality is equal to the number given in each case. On the $E_2$-page, the nontrivial entries $E_{2}^{s,t}$ such that $s+t=6, t>1$ are $E_{2}^{4,2}\cong\mathbb{Z}/m$, $E_{2}^{2,4}\cong\mathbb{Z}/m$ and $E_{2}^{0,6}\cong\mathbb{Z}/m$.

For $E_{2}^{0,6}$, we have
$$d_{3}^{0,6}: E_{2}^{0,6}\cong\mathbb{Z}/m\rightarrow E_{2}^{3,4}\cong\mathbb{Z}/m$$
the  multiplication by $3n$, by the Leibniz rule. Hence we have
\begin{equation}\label{E^0,6}
E_{4}^{0,6}=\operatorname{Ker}d_3^{0,6}\cong\mathbb{Z}/\epsilon_3(\frac{m}{n})n.
\end{equation}

We argue case by case, as follows.

\textbf{Case 1}: $n$ is even, or $m,n$ are both odd. In this case, either $4|mn$, or $mn$ is odd. Then it follows from Theorem \ref{Cohomology of BPU} that
$$E_{2}^{7,0}\cong H^7(\mathbf{B}PU_{mn};\mathbb{Z})=0.$$
Then for degree reasons there is no nontrivial differential into or out of $E_2^{4,2}$. Hence we have
\begin{equation}\label{E^4,2 case1}
E_{\infty}^{4,2}\cong E_2^{4,2}\cong\mathbb{Z}/m.
\end{equation}

Next we consider the entry $E_{*}^{2,4}$. For degree reasons the only possibly nontrivial differential into or out of it is $d_{3}^{2,4}$, of which the codomain, according to the K{\"u}nneth's theorem and Theorem \ref{Cohomology of BPU}, is
\begin{equation*}
E_{3}^{5,2}\cong H^5(\mathbf{B}PU_{mn};\mathbb{Z}/m)\cong
\begin{cases}
0,\textrm{ if $n$ is odd, and consequentely so is $m$,}\\
\mathbb{Z}/2, \textrm{ if $n$ is even}.
\end{cases}
\end{equation*}
We proceed to show $d_{3}^{2,4}=0$ in both cases. When $n$ is odd this is obvious. When $n$ is even, it follows from Lemma \ref{noncyclic} that $E_{\infty}^{5,2}\neq 0$, for otherwise $H^7(\mathbf{B}P(mn,n);\mathbb{Z})$ would be cyclic, a contradiction. Therefore, $E_{\infty}^{5,2}=E_{2}^{5,2}\cong\mathbb{Z}/2$, from which it follows that $d_{3}^{2,4}=0$. Hence, we have
\begin{equation}\label{E^2,4 case1}
E_{\infty}^{2,4}\cong E_2^{2,4}\cong\mathbb{Z}/m.
\end{equation}
Again for degree reasons, the only potentially nontrivial differential out of $E_4^{0,6}$ is into $E_4^{5,2}\cong\mathbb{Z}/2$. It follows from Lemma \ref{noncyclic} that this differential is $0$. Therefore we have
\begin{equation}\label{E^0,6 case1}
E_{\infty}^{0,6}\cong E_4^{0,6}\cong\mathbb{Z}/\epsilon_3(\frac{m}{n})n.
\end{equation}
Case 1 now follow from (\ref{E^4,2 case1}), (\ref{E^2,4 case1}) and (\ref{E^0,6 case1}).

\textbf{Case 2}: $n$ is odd, and $4|m$. In this case it again follows from Theorem \ref{Cohomology of BPU} that $H^{7}(\mathbf{B}P_{mn};\mathbb{Z})=0$. Then for the same reason as in Case 1 we have
\begin{equation}\label{E^4,2 case2}
E_{\infty}^{4,2}\cong E_2^{4,2}\cong\mathbb{Z}/m.
\end{equation}
Consider $E_*^{2,4}$. The only potentially nontrivial differential into or out of $E_*^{2,4}$ is
$$d_3^{2,4}: E_3^{2,4}\cong\mathbb{Z}/m\rightarrow E_3^{5,2}\cong\mathbb{Z}/2,$$
where $E_3^{5,2}\cong H^5(\mathbf{B}P_{mn};\mathbb{Z}/m)\cong\mathbb{Z}/2$ follows from K{\"u}nneth's theorem.
Since $n$ is odd, it follows from Corollary \ref{p-tor in E} that $E_{\infty}^{5,2}=0$. However, for degree reasons there is no nontrivial differential into or out of $E_*^{5,2}$ except for possibly $d_3^{2,4}$. (Notice that $E_{2}^{8,0}$ has no $2$-torsion, by Theorem \ref{Cohomology of BPU}.) Therefore, $d_3^{2,4}$ is surjective and it follows that
\begin{equation}\label{E^2,4 case2}
E_{\infty}^{2,4}\cong\mathbb{Z}/\frac{m}{2}.
\end{equation}

For degree reasons and the fact that $d_3^{3,4}$ is surjective, there is no nontrivial entry of total degree $7$ on the $E_4$-page. Hence it follows that
\begin{equation}\label{E^0,6 case2}
E_{\infty}^{0,6}\cong E_4^{0,6}\cong\mathbb{Z}/\epsilon_3(\frac{m}{n})n.
\end{equation}
Case 2 then follows from (\ref{E^4,2 case2}), (\ref{E^2,4 case2}) and (\ref{E^0,6 case2}).

\textbf{Case 3}: $n$ is odd, and $m=2(2l+1)$ for some integer $l$. In this case it follows from Theorem \ref{Cohomology of BPU} that
$$E_{2}^{7,0}\cong H^7(\mathbf{B}PU_{mn};\mathbb{Z})\cong\mathbb{Z}/2,$$
and moreover, the differential
$$d_3^{4,2}: E_3^{4,2}\cong\mathbb{Z}/m\rightarrow E_{2}^{8,0}\cong\mathbb{Z}/2$$
is onto, since, due to Theorem \ref{Cohomology of BPU}, $H^7(\mathbf{B}PU_{mn};\mathbb{Z})$ is generated by the cup product $e_2x_1$. For degree reasons there is no other nontrivial differentials into or out of $E_3^{4,2}$. Hence we have
\begin{equation}\label{E^4,2 case3}
E_{\infty}^{4,2}\cong\mathbb{Z}/\frac{m}{2}.
\end{equation}
For $E_{\infty}^{2,4}$ and $E_{\infty}^{0,6}$ the same arguments as in Case 2 applies and we have
\begin{equation}\label{E^2,4 case3}
E_{\infty}^{4,2}\cong\mathbb{Z}/\frac{m}{2}
\end{equation}
and
\begin{equation}\label{E^0,6 case3}
E_{\infty}^{0,6}\cong\mathbb{Z}/\epsilon_3(\frac{m}{n})n.
\end{equation}
Therefore, Case 3 follows.
\end{proof}

The study of the next non-trivial stage of the Postnikov tower requires some auxiliary results on the cohomology of the classifying spaces of some Lie groups, which is the topic of the next section.

\section{the cohomology of classifying spaces of some lie groups}
In \cite{Gu}, the author considered the integral cohomological Serre spectral sequence associated to the fiber sequence $\mathbf{B}U_r\rightarrow\mathbf{B}PU_r\rightarrow K(\mathbb{Z},3)$, which we denote by $^{U}E_{*}^{*,*}$.
and found a formula for the differential $^{U}d_3$. For degree reasons we have
$$^{U}E_{3}^{s,t}\cong {^UE}_{2}^{s,t}\cong H^{s}(K(\mathbb{Z},3);H^{t}(\mathbf{B}U_r;\mathbb{Z})).$$
Let $c_k\in H^{2k}(\mathbf{B}U_r;\mathbb{Z})$ be the $k$th Chern class, and $x_1$ be the generator of $H^{3}(K(\mathbb{Z},3);\mathbb{Z})$. Then we have the following
\begin{proposition}[Corollary 5.3, \cite{Gu}]
 $^{U}d_{3}(c_{k})=(r-k+1)c_{k-1}x_1.$
\end{proposition}

In low dimensions, for example, $^{U}E_{3}^{0,4}$ and $^{U}E_{3}^{0,6}$, $^{U}d_{3}$ is the only non-trivial differential out of them. Therefore, the kernel of $^{U}d_{3}^{0,*}$ gives the image of the homomorphism $H^{*}(\mathbf{B}PU_r;\mathbb{Z})\rightarrow H^{*}(\mathbf{B}U_r;\mathbb{Z})$ induced by the quotient map $\mathbf{B}U_r\rightarrow\mathbf{B}PU_r$. A straightforward calculation gives the following

\begin{lemma}\label{BU to BPU}
The image of the homomorphism $H^{*}(\mathbf{B}PU_r;\mathbb{Z})\rightarrow H^{*}(\mathbf{B}U_r;\mathbb{Z})$ in degree $4$ and $6$ are generated respectively by $$\epsilon_{2}(r)(rc_2-\frac{r-1}{2}c_1^2)$$
and
$$\frac{\epsilon_{3}(r)}{\epsilon_{2}(r)\epsilon_2(\frac{r-2}{\epsilon_{2}(r-2)})}\big[r^2c_3-r(r-2)c_1c_2+
\frac{(r-1)(r-2)}{3}c_1^3\big].$$
\end{lemma}

By pre-composing the quotient map with the inclusion $SU_r\hookrightarrow U_r$, we obtain another quotient map $SU_r\rightarrow PU_r$. Applying the classifying space functor and taking integral cohomology, we obtain the homomorphism
$$H^{*}(\mathbf{B}PU_r;\mathbb{Z})\rightarrow H^{*}(\mathbf{B}SU_r;\mathbb{Z}).$$
Recall that the inclusion $SU_r\hookrightarrow U_r$ induces a homomorphism
$$H^{*}(\mathbf{B}U_r;\mathbb{Z})\cong\mathbb{Z}[c_1,\cdots,c_n]\rightarrow H^{*}(\mathbf{B}SU_r;\mathbb{Z})\cong
\mathbb{Z}[c_2,\cdots,c_n]$$
which annihilates $c_1$ and takes $c_i$ to itself, for $i>1$. Therefore, Lemma \ref{BU to BPU} immediately implies the following
\begin{lemma}\label{BSU to BPU}
The image of the homomorphism
$$H^{*}(\mathbf{B}PU_r;\mathbb{Z})\rightarrow H^{*}(\mathbf{B}SU_r;\mathbb{Z})$$
in degree $4$ and $6$ are generated respectively by $\epsilon_{2}(r)rc_2$ and
$$\frac{\epsilon_{3}(r)r^2}{\epsilon_{2}(r)\epsilon_2(\frac{r-2}{\epsilon_{2}(r-2)})}c_3.$$
\end{lemma}

We conclude this section with the following
\begin{proposition}\label{H^6 BP(n,mn) to H^6 BSU_mn}
Let $n$ and $m$ be such that $\epsilon_{2}(n)n|m$. Consider the quotient map $SU_{mn}\rightarrow P(n,mn)$. The induced homomorphism
$$H^{6}(\mathbf{B}P(n,mn);\mathbb{Z})\rightarrow H^{6}(\mathbf{B}SU_{mn};\mathbb{Z})=\mathbb{Z}[c_3]$$
has image generated by
$$\frac{\epsilon_3(mn)}{\epsilon_3(m/n)\epsilon_{2}(n)}nc_3.$$
\end{proposition}
\begin{proof}
Notice that the quotient map $SU_{mn}\rightarrow PU_{mn}$ can be factorized as $SU_{mn}\rightarrow P(n,mn)\rightarrow PU_{mn}$, and the proposition follows from Lemma \ref{BSU to BPU} and Corollary \ref{H^6(BP)}, once we notice the following equation:
\begin{equation*}
\frac{\epsilon_3(mn)m^2n^2}{\epsilon_2(mn)\epsilon_2(\frac{mn-2}{\epsilon_2(mn-2)})}=
\begin{cases}
\frac{\epsilon_3(mn)m^2n^2}{\epsilon_2(n)}, \textrm{ $n$ is even(hence so is $m$), or $m,n$ are odd,}\\
\frac{\epsilon_3(mn)m^2n^2}{2\epsilon_2(n)}, \textrm{ $n$ is odd, and $4|m$,}\\
\frac{\epsilon_3(mn)m^2n^2}{4\epsilon_2(n)}, \textrm{ $n$ is odd, and $m=2(2l+1)$ for some $l$.}\\
\end{cases}
\end{equation*}
\end{proof}
\section{Proof of Theorem \ref{main1}}
We proceed to consider $H^{7}(\mathbf{B}P(n,mn);\mathbb{Z})$. Consider the fiber sequence
$$\mathbf{B}SU_{mn}\rightarrow\mathbf{B}P(n,mn)\rightarrow K(\mathbb{Z}/n,2)$$
and the associated integral cohomological Serre spectral sequence
${^{S}E}_{*}^{*,*}$ with
$${^{S}E}_{2}^{s,t}\cong H^{s}(K(\mathbb{Z}/n,2);H^{t}(\mathbf{B}SU_{mn};\mathbb{Z})).$$
We have the following
\begin{lemma}\label{^S d_3}
Suppose that $\epsilon_{2}(n)n|m$. Recall that $H^{3}(K(\mathbb{Z}/n,2);\mathbb{Z}))\cong\mathbb{Z}/n$ is generated by an element $\beta_n$, and that $H^{7}(K(\mathbb{Z}/n,2);\mathbb{Z}))\cong\mathbb{Z}/\epsilon_{3}(n)n$ is generated by $R_n$. In the spectral sequence ${^{S}E}_{*}^{*,*}$, we have ${^{S}d}_{3}^{0,6}(c_3)=2c_{2}\beta_{n}$ with kernel generated by
$$\frac{n}{\epsilon_{2}(n)}c_3,$$
and
$${^{S}d}_{7}^{0,6}(\frac{n}{\epsilon_{2}(n)}c_3)=
\frac{\epsilon_{3}(n)\epsilon_3(m/n)}{\epsilon_3(mn)}nR_n.$$
All the other differentials out of ${^{S}E}_{*}^{0,6}$ are trivial.

In particular ${^{S}d}_{3}^{0,6}$ is the only non-trivial differential out of ${^{S}E}_{*}^{0,6}$ when $\epsilon_3(n)n|m$.
\end{lemma}
\begin{proof}
See Figure \ref{^S E figure} for the differentials of the spectral sequence ${^{S}E}_{*}^{*,*}$ discussed here. For degree reasons the only potentially non-trivial differential out of ${^{S}E}_{2}^{0,4}$ is ${^{S}d}_{5}^{0,4}$. By Proposition \ref{H^5(BP)} we have
$${^{S}E}_{\infty}^{5,0}\cong\mathbb{Z}/\epsilon_{2}(n)n\cong{^{S}E}_{2}^{5,0},$$
from which it follows that ${^{S}d}_{5}^{0,4}=0$. The statement about ${^{S}d}_{3}^{0,6}$ follows from an easy comparison of the cohomological Serre spectral sequences between the fiber sequences
$$\mathbf{B}SU_{mn}\rightarrow\mathbf{B}P(n,mn)\rightarrow K(\mathbb{Z}/n,2)$$
and
$$\mathbf{B}U_{mn}\rightarrow\mathbf{B}PU_{mn}\rightarrow K(\mathbb{Z},3),$$
the $E_3$-page of the latter of which is well understood in \cite{Gu}. Therefore, the statement about ${^{S}d}_{7}^{0,6}$ follows from Proposition \ref{H^6 BP(n,mn) to H^6 BSU_mn}.
\end{proof}

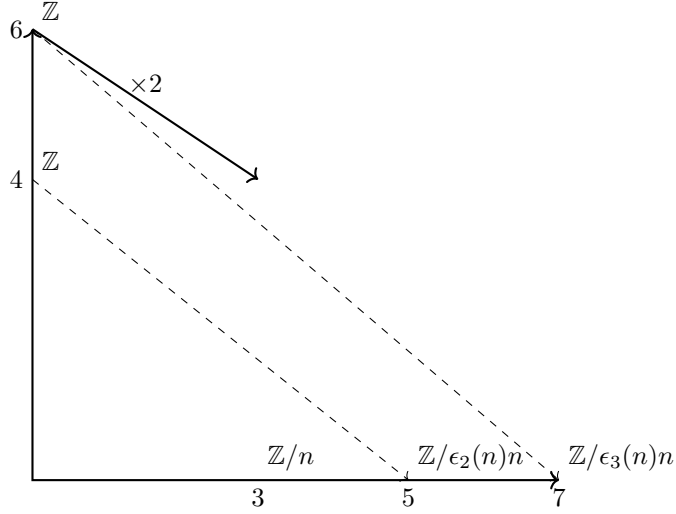
\begin{figure}[!]
\begin{tikzpicture}
\draw[step=0.5cm,gray, thick];
\draw [ thick, <->] (0,6)--(0,0)--(7,0);

\node [below] at (3,0) {$3$};
\node [below] at (5,0) {$5$};
\node [below] at (7,0) {$7$};
\node [left] at (0,4) {$4$};
\node [left] at (0,6) {$6$};
\node [above right] at (0,4) {$\mathbb{Z}$};
\node [above right] at (0,6) {$\mathbb{Z}$};
\node [above right] at (3,0) {$\mathbb{Z}/n$};
\node [above right] at (5,0) {$\mathbb{Z}/\epsilon_{2}(n)n$};
\node [above right] at (7,0) {$\mathbb{Z}/\epsilon_{3}(n)n$};
\draw (0,4) [dashed, ->] to (5,0);
\draw (0,6) [thick, ->] to node [above] {$\times 2$}(3,4);
\draw (0,6) [dashed, ->] to (7,0);
\end{tikzpicture}
\caption{Low dimensional differentials of the spectral sequence $^{S}E_{*}^{*,*}$, when $\epsilon_3(n)n|m$. The dashed arrows represent trivial differentials.}\label{^S E figure}
\end{figure}

For future convenience we introduce the following notation:
\begin{equation}\label{I(m,n)}
I(m,n)=\frac{\epsilon_3(m/n)\epsilon_3(n)}{\epsilon_3(mn)}n.
\end{equation}
Lemma \ref{^S d_3} has the following immediate consequence:
\begin{corollary}\label{H^7(BP(n,mn)) n odd}
Assume that $n$ is odd, and $n|m$. Then we have
$$H^{7}(\mathbf{B}P(n,mn);\mathbb{Z})\cong\mathbb{Z}/I(m,n),$$
which is generated by $R_n(x_1'')$. Here $x_1''$ generates $H^{2}(\mathbf{B}P(n,mn);\mathbb{Z}/n)$. Moreover, $B(x_1'')=x_1'$ where $B$ is the Bockstein homomorphism.
\end{corollary}

The general case is more complicated. Recall Lemma \ref{noncyclic}, which says that $H^{7}(\mathbf{B}P(n,mn);\mathbb{Z})$ is not a cyclic group when $n$ is even. With a little more work we can impose a strong restriction on the $k$-invariant $\kappa_5$.
\begin{remark}\label{identification}
 Since we have the homotopy equivalence
 $$\mathbf{B}P(n,mn)[4]=\mathbf{B}P(n,mn)[5]\simeq K(\mathbb{Z}/n,2)\times K(\mathbb{Z},4),$$
 the induced homomorphism
 $$H^{*}(K(\mathbb{Z}/n,2)\times K(\mathbb{Z},4);\mathbb{Z})\rightarrow H^{*}(\mathbf{B}P(n,mn);\mathbb{Z})$$
 is an isomorphism  in degree $\leq 6$ and an epimorphism in degree $7$. In view of this, in what follows we do not explicitly distinguish the elements $x_1', R_n(x_1''), e_2'$ with $\beta_n, R_n, \iota_4$, but stay aware of the relations in degree $7$.
\end{remark}
\begin{lemma}\label{kappa_5}
Assume $\epsilon_{2}(n)n|m$. Then we have
$$\kappa_5=\lambda_1 R_n\times 1+\lambda_2\beta_n\times\iota_4+1\times\Gamma_4$$
where $\lambda_1\in\mathbb{Z}/\epsilon_{3}(n)n$, $\lambda_2\in\mathbb{Z}/n$. Furthermore, the subgroups of $\mathbb{Z}/n$ generated by $2\lambda_2$ contains $2$.

In particular, if $n$ is odd, then, up to a scalar multiple, we have
$$\kappa_5=\lambda_1 R_n\times 1+\beta_n\times\iota_4+1\times\Gamma_4$$
\end{lemma}
\begin{proof}
Suppose
$$\kappa_5=\lambda_1 R_n\times 1+\lambda_{2}\beta_n\times\iota_4+\lambda_{3}\times\Gamma_4,$$
where $\lambda_1\in \mathbb{Z}/\epsilon_{3}(n)n, \lambda_2\in\mathbb{Z}/n, \lambda_3\in\mathbb{Z}/2$. Lemma \ref{^S d_3} implies that $H^{7}(\mathbf{B}P(n,mn);\mathbb{Z})$ is generated by $R_n\times 1$ and $\beta_n\times\iota_4$, since they generate $^{S}E_{2}^{3,4}$ and $^{S}E_{2}^{7,0}$, the only non-trivial entries on the $E_2$-page with total degree $7$. In particular, the class $1\times\Gamma_4$ is a linear combination of them, from which it follows that $\lambda_3=1$.

Let $\langle\kappa_5\rangle$ be the subgroup of $H^{7}(\mathbf{B}P(n,mn);\mathbb{Z})$ generated by $\kappa_5$. Lemma \ref{^S d_3} implies that $2\beta_n\times\iota_4$ is in the subgroup generated by $R_n\times 1$, since $^{S}E_{\infty}^{3,4}$, generated by $\beta_n\times\iota_4$, has order $2$. Hence, there is some scalar $\Lambda\in\mathbb{Z}/\epsilon_{3}(n)n$ such that
\begin{equation}\label{kappa_5 subgroup}
2\beta_n\times\iota_4+\Lambda R_n\times 1\in\langle\kappa_5\rangle.
\end{equation}

Let $\langle\beta_n\times\iota_4, R_n\times 1\rangle$ be the subgroup of $H^{7}(\mathbf{B}P(n,mn);\mathbb{Z})$ generated by $\beta_n\times\iota_4$ and $R_n\times 1$. Then (\ref{kappa_5 subgroup}) implies
$$2\beta_n\times\iota_4+\Lambda R_n\times 1\in\langle\kappa_5\rangle\cap\langle\beta_n\times\iota_4, R_n\times 1\rangle=\langle 2\kappa_5\rangle,$$
where the identification of subgroups follows from the fact $$2\kappa_5\in\langle\kappa_5\rangle\cap\langle\beta_n\times\iota_4, R_n\times 1\rangle$$
and that $2$ is a prime number.

From the above, it follows that $2\beta_n\times\iota_4+\Lambda R_n\times 1$ is a multiple of
$$2\kappa_5=2(\lambda_1 R_n\times 1+\lambda_{2}\beta_n\times\iota_4+\lambda_{3}\times\Gamma_4)=2(\lambda_1 R_n\times 1+\lambda_{2}\beta_n\times\iota_4),$$
which implies that the ideal of $\mathbb{Z}/n$ generated by $2\lambda_2$ contains $2$, and the lemma follows.
\end{proof}

We make the following important observation:
\begin{lemma}\label{e_3(n)n}
Let $X$ be the $8$-skeleton of $K(\mathbb{Z}/n,2)$, with Brauer class
$$\alpha\in H^{3}(X;\mathbb{Z})_{\mathrm{tor}}=\operatorname{Br}(X)$$
the restriction of the fundamental class $\beta_n\in H^{3}(K(\mathbb{Z}/n,2);\mathbb{Z})$. If $\alpha$ is classified by a $PU_{mn}$-torsor, then $\epsilon_2(n)\epsilon_3(n)n|m$.
\end{lemma}
\begin{proof}
For the obvious reason we do not distinguish cohomology classes of $X$ and $K(\mathbb{Z}/n,2)$ in degree $\leq 7$. It follows from Proposition \ref{kappa3} that $\epsilon_2(n)n|m$. It remains to prove that $\epsilon_3(n)n|m$. Assume $\epsilon_3(n)n\nmid m$ for a contradiction. Since we already have $\epsilon_2(n)n|m$, it follows that $3|n$ and $3\nmid\frac{m}{n}$. Hence we have
\begin{equation}\label{nR_n}
I(m,n)=n.
\end{equation}

Since $H^{2}(X;\mathbb{Z})=0$, there is a unique element $\alpha'\in H^{2}(X;\mathbb{Z}/n)$ such that $B(\alpha')=\alpha$, where $B$ is the Bockstein homomorphism. Therefore, the lifting problem shown by the following diagram
\begin{equation*}
\begin{tikzcd}
&\mathbf{B}P(n,mn)[5]=K(\mathbb{Z},4)\times K(\mathbb{Z}/n,2)\arrow{d}\arrow[r,"\kappa_5"]&K(\mathbb{Z},7)\\
X\arrow[r,"\alpha'"]\arrow[ru,dashed,"f_5"]&\mathbf{B}P(n,mn)[3]=K(\mathbb{Z}/n,2)
\end{tikzcd}
\end{equation*}
has a unique solution $f_5$ since $H^4(X;\mathbb{Z})=0$, and for the same reason, the composition $\kappa_5\cdot f_5$ is $\lambda_1 R_n\in H^7(X;\mathbb{Z})$, for some $\lambda_1\in\mathbb{Z}/\epsilon_3(n)n$, according to Lemma \ref{kappa_5}. On the other hand, it follows from Lemma \ref{^S d_3} and (\ref{nR_n}) that in $H^7(\mathbf{B}P(n,mn);\mathbb{Z})$ we have $nR_n=0$, whereas $R_n$ is of degree $3n$ in $H^7(\mathbf{B}P(n,mn)[5];\mathbb{Z})$. Therefore, we have
$$0\neq nR_n\in\langle\kappa_5\rangle\in H^7(\mathbf{B}P(n,mn)[5];\mathbb{Z}).$$
This implies that $nR_n$ is a multiple of $\lambda_1 R_n$, and in particular, $\kappa_5\cdot f_5=\lambda_1 R_n\neq 0$. Hence $f_5$ does not lift to $\mathbf{B}P(n,mn)[6]$, a contradiction.
\end{proof}
We proceed to study the next non-trivial stage of the Postnikov tower of $\mathbf{B}P(n,mn)$, namely $\mathbf{B}P(n,mn)[6]$. Recall from Lemma \ref{noncyclic} that when $n$ is even, we have
\begin{equation*}
\begin{split}
&H^{7}(\mathbf{B}P(n,mn);\mathbb{Z})\cong H^{7}(\mathbf{B}P(n,mn)[6];\mathbb{Z})\\
\cong &H^{7}(K(\mathbb{Z},4)\times K(\mathbb{Z}/n,2);\mathbb{Z})/(\kappa_ 5)\cong\mathbb{Z}/\epsilon_{3}(n)n\oplus\mathbb{Z}/n\oplus\mathbb{Z}/2/(\kappa_5),
\end{split}
\end{equation*}
where the components $\mathbb{Z}/\epsilon_{3}(n)n$, $\mathbb{Z}/n$ and $\mathbb{Z}/2$ are generated by $R_n\times 1$, $\beta_n\times\iota_4$ and $1\times\Gamma_4$.
\begin{lemma}\label{E_infty^3,4}
Suppose $\epsilon_3(n)n|m$. Then we have $E_{\infty}^{3,4}\cong\mathbb{Z}/\epsilon_3(n)n$. Moreover, as a direct sum component of $H^7(\mathbf{B}P(n,mn);\mathbb{Z})$, $E_{\infty}^{3,4}$ is generated by $R_n$.
\end{lemma}
\begin{proof}
Consider the following commutative diagram
\begin{equation}\label{ladder of fib}
\begin{tikzcd}
K(\mathbb{Z}/m,1)\arrow[d,"="]\arrow{r}&\mathbf{B}P(n,mn)\arrow{d}\arrow{r}&\mathbf{B}PU_{mn}\arrow{d}\\
K(\mathbb{Z}/m,1)\arrow{r}&K(\mathbb{Z}/n,2)\arrow{r}&K(\mathbb{Z}/mn,2)
\end{tikzcd}
\end{equation}
It follows from the proof of Lemma \ref{R_n} that the morphism of spectral sequences induced by (\ref{ladder of fib}) restricts to an isomorphism
$${^H}E_4^{3,4}\cong E_4^{3,4}.$$
We proceed to show that it restricts to an isomorphism
\begin{equation}\label{Einfty(3,4)}
{^H}E_{\infty}^{3,4}\cong E_{\infty}^{3,4}.
\end{equation}
For degree reasons, $d_{3}^{0,6}$ is the only non-trivial differential reaching the entry $E_{*}^{3,4}$. By the Leibniz rule, we have
$$E_{3}^{3,4}/\operatorname{Im}d_{3}^{0,6}\cong(\mathbb{Z}/m)/3n\cong\mathbb{Z}/\epsilon_3(m)n,$$
where the last equation follows from the fact that $\epsilon_{3}(n)n|m$. On the other hand, it follows from Leibniz rule that $d_3^{3,4}(v^2x_1)=2nvx_1=0$. Therefore we have $d_3^{3,4}=0$ and $$E_{4}^{3,4}=E_{4}^{3,4}\cong\mathbb{Z}/\epsilon_3(m)n.$$
For degree reasons, the only potentially nontrivial differential into or out of $E_{4}^{3,4}$ is
$$d_5^{3,4}: E_{5}^{3,4}\rightarrow E_{5}^{8,0}.$$
It then follows from (\ref{Hd_5(3,4)}) in the proof of Lemma \ref{R_n} that $d_5^{3,4}=0$ when $3|n$. On the other hand, it follows from Corollary \ref{d_5} that
$$d_5^{3,4}: E_{5}^{3,4}\cong\mathbb{Z}/\epsilon_3(m)n\rightarrow E_{5}^{8,0}\cong H^8(\mathbf{B}PU_{mn};\mathbb{Z})$$
is onto the subgroup of $E_5^{8,0}$ of order $3$ when $3\nmid n$ and $3|m$. Therefore, (\ref{Einfty(3,4)}) holds, and the desired assertion follows from Lemma \ref{R_n}.
\end{proof}

The next lemma concerns $\beta_n\times\iota_4$, which, as an element of $H^{7}(\mathbf{B}P(n,mn);\mathbb{Z})$, is identified with $e_2'x_1'$. Recall that $E_{2}^{0,4}\cong\mathbb{Z}/m$ is generated by $v^2$.
\begin{lemma}\label{E_infty^0,4}
The group $E_{\infty}^{0,4}\cong\mathbb{Z}/\epsilon_2(m)n$ is the subgroup of $E_{2}^{0,4}$ generated by $\frac{m}{\epsilon_2(m)n}v^2$.As a quotient of $H^{4}(\mathbf{B}P(n,mn);\mathbb{Z})$, $E_{\infty}^{0,4}$ is generated  by the element represented by $e_2'$, the generator of $H^{4}(\mathbf{B}P(n,mn);\mathbb{Z})$. Moreover, suppose that $\epsilon_2(n)\epsilon_3(n)n|m$. Then we can choose $v$, up to an invertible scalar coefficient, so that in $E_{\infty}^{0,4}$ there is a relation $e_2'=\frac{m}{\epsilon_2(m)n}v^2$.
\end{lemma}
\begin{proof}
The fact that $E_{\infty}^{0,4}\cong\mathbb{Z}/\epsilon_2(m)n$ and that it is generated by $\frac{m}{\epsilon_2(n)n}v^2$ follows from Lemma \ref{d_3} and the Leibniz rule. The rest follows from Corollary \ref{H^4(BP)}.
\end{proof}
\begin{theorem}\label{H^7(BP)}
Suppose that $\epsilon_2(n)\epsilon_3(n)n|m$. Then
\begin{equation}\label{H^7 odd even}
H^{7}(\mathbf{B}P(n,mn);\mathbb{Z})\cong
\begin{cases}
\mathbb{Z}/\epsilon_3(n)n,\quad\textrm{$n$ odd,}\\
\mathbb{Z}/2\oplus\mathbb{Z}/\epsilon_3(n)n, \quad\textrm{$n$ even.}
\end{cases}
\end{equation}
Furthermore,
\begin{enumerate}
\item
if $n$ is odd, then, up to an invertible scalar,
$$\kappa_5=\frac{\epsilon_3(n)m}{\epsilon_3(m)n}\lambda R_n\times 1+\beta_n\times\iota_4+1\times\Gamma_4$$
\item
if $n$ is even, then, up to an invertible scalar,
$$\kappa_5=\frac{\epsilon_3(n)m}{\epsilon_2(m)\epsilon_3(m)n}\lambda R_n\times 1+\lambda_2\beta_n\times\iota_4+1\times\Gamma_4\quad \mathrm{mod}\quad 2\mathrm{-torsion}.$$
where $\lambda\in\mathbb{Z}/\epsilon_3(n)n$ is invertible and $\lambda_2$ is as in Lemma \ref{kappa_5}.
\end{enumerate}
\end{theorem}

\begin{proof}
The first case of (\ref{H^7 odd even}) is just Corollary \ref{H^7(BP(n,mn)) n odd}. For the other case, it follows from Lemma \ref{E_infty^3,4} and  Corollary \ref{H^7(BP(n,mn)) n odd} that we have a short exact sequence
$$0\rightarrow\mathbb{Z}/\epsilon_3(n)n\rightarrow H^{7}(\mathbf{B}P(n,mn);\mathbb{Z})\rightarrow\mathbb{Z}/2\rightarrow 0.$$
By Lemma \ref{noncyclic}, the group $H^{7}(\mathbf{B}P(n,mn);\mathbb{Z})$ is not cyclic. Therefore we have $$H^{7}(\mathbf{B}P(n,mn);\mathbb{Z})\cong
\mathbb{Z}/2\oplus\mathbb{Z}/\epsilon_3(n)n,$$
for $n$ even.

To prove (1) and (2) we need to find the coefficient $\lambda_1$ as in Lemma \ref{kappa_5}. This is accomplished by studying the element $e_2'x_1'\in H^{7}(\mathbf{B}P(n,mn);\mathbb{Z})$. In particular, we locate it in the spectral sequence $E_{*}^{*,*}$.

It follows from Lemma \ref{E_infty^0,4} that in $E_3^{3,4}$ we have the relation $e_2'x_1'=\frac{m}{\epsilon_2(m)n}v^2x_1$. On the other hand, it follows from Lemma \ref{E_infty^3,4} that $E_{\infty}^{3,4}$ is generated by $\frac{\epsilon_3(m)}{\epsilon_3(n)}v^2x_1$ which is identified with $\lambda R_n$ for some invertible element $\lambda\in\mathbb{Z}/\epsilon_3(n)n$. Hence, in $E_{\infty}^{3,4}$ we have the relation
\begin{equation}\label{kappa_5 in E}
e_2'x_1'=\frac{\epsilon_3(n)m}{\epsilon_2(m)\epsilon_3(m)n}\lambda R_n
\end{equation}
Since we have
\begin{equation*}
H^{7}(\mathbf{B}P(n,mn);\mathbb{Z})\cong
\begin{cases}
E_{\infty}^{3,4},\quad\textrm{$n$ odd,}\\
E_{\infty}^{3,4}\oplus E_{\infty}^{5,2}\cong E_{\infty}^{3,4}\oplus\mathbb{Z}/2,\quad\textrm{$n$ even,}
\end{cases}
\end{equation*}
The desired statement (1),(2) then follows immediately from (\ref{kappa_5 in E}).
\end{proof}

\begin{proof}[Proof of Theorem \ref{main1}]
Let $X$ be an $8$-complex and $\alpha\in\operatorname{Br}(X)$. The first paragraph of the theorem, that $\operatorname{ind}(\alpha)|\epsilon_2(n)\epsilon_3(n)n^3$ follows immediately from Corollary \ref{upper bound}.

We proceed to prove the second paragraph. Let $X$, $\alpha$ and $\alpha'$ be as in Lemma \ref{e_3(n)n}. Consider $\alpha'$ as a map $X\rightarrow K(\mathbb{Z}/n,2)$. Then it follows from Proposition \ref{kappa3}, that $\alpha'$ has a lift to $\mathbf{B}P(n,mn)[4]$ if and only if $\epsilon_2(n)n|m$, in which case
$$\mathbf{B}P(n,mn)[4]\simeq K(\mathbb{Z}/n,2)\times K(\mathbb{Z},4)$$
and we have a unique lift $f_5$ as in the following diagram
\begin{equation}\label{main1diag}
\begin{tikzcd}
&\mathbf{B}P(n,mn)[5]=K(\mathbb{Z},4)\times K(\mathbb{Z}/n,2)\arrow{d}\arrow[r,"\kappa_5"]&K(\mathbb{Z},7)\\
X\arrow[r,"\alpha'"]\arrow[ru,dashed,"f_5"]&\mathbf{B}P(n,mn)[3]=K(\mathbb{Z}/n,2)
\end{tikzcd}
\end{equation}
as discussed in Lemma \ref{e_3(n)n}. Therefore, a lift of $\alpha_n'$ to $\mathbf{B}P(n,mn)[5]$ exists if and only if
$$\kappa_5\cdot f_5=0\in H^7(X;\mathbb{Z})\cong H^7(K(\mathbb{Z}/n,2);\mathbb{Z}).$$
On the other hand, the projection
$$\mathbf{B}P(n,mn)[4]\simeq K(\mathbb{Z}/n,2)\times K(\mathbb{Z},4)\rightarrow K(\mathbb{Z}/n,2)$$
splits, from which it follows that the homomorphism $H^*(f_5)$ is a quotient homomorphism sending exactly all classes in $H^*(K(\mathbb{Z},4);\mathbb{Z})$ to $0$. Recall from Proposition \ref{kappa3} and Lemma \ref{e_3(n)n} that we only need to consider $m$ and $n$ such that $\epsilon_2(n)\epsilon_3(n)n|m$.
Therefore, it follows from Theorem \ref{H^7(BP)} that
\begin{equation}\label{kappa_5f_5}
\kappa_5\cdot f_5=f_5^*(\kappa_5)=
\begin{cases}
\frac{\epsilon_3(n)m}{\epsilon_3(m)n}\lambda R_n, \quad\textrm{$n$ odd,}\\
\frac{\epsilon_3(n)m}{\epsilon_2(m)\epsilon_3(m)n}\lambda R_n\quad\mathrm{mod}\quad 2\mathrm{-torsion}, \quad\textrm{$n$ even.}
\end{cases}
\end{equation}
When $n$ is odd, $\kappa_5\cdot f_5=0$ if and only if $\epsilon_3(n)n|\frac{\epsilon_3(n)m}{\epsilon_3(m)n}$, i.e., $\epsilon_3(m)n^3|mn$. Since $n|m$, we have $\epsilon_3(n)|\epsilon_3(m)$, whence $\epsilon_3(n)n^3|mn$, for all $m,n$ such that the lift of $\alpha'$ to $\mathbf{B}P(n,mn)[6]$ exists, which, according to Proposition \ref{lift to BP}, is true if and only if $\alpha$ can be classified by a $PU_{mn}$-torsor. Therefore, $\epsilon_3(n)n^3|\operatorname{ind}(\alpha)$. Then it follows that $\epsilon_3(n)n^3=\operatorname{ind}(\alpha)$, as desired.

When $n$ is even, the same argument can be made with the indeterminacy of $2$-torsions:
$$\kappa_5\cdot f_5=\frac{\epsilon_3(n)m}{\epsilon_2(m)\epsilon_3(m)n}\lambda R_n=0\quad\mathrm{mod}\quad 2\mathrm{-torsions},$$
where $\lambda\in\mathbb{Z}/\epsilon_3(n)n$ is invertible. In other words, we have $\frac{\epsilon_3(n)n}{2}|\frac{\epsilon_3(n)m}{\epsilon_2(m)\epsilon_3(m)n},$
i.e., $\epsilon_3(m)n^3|mn$. Since $n|m$, we have $\epsilon_3(n)n^3|mn$ for all $m,n$ such that $\alpha$ is classified by a $PU_{mn}$-torsor. Hence $\epsilon_3(n)n^3|\operatorname{ind}(\alpha)$.
\end{proof}
\section{Proof of Theorem \ref{main2}}
It follows from Proposition \ref{kappa3} that $\mathbf{B}P(n,mn)[5]\simeq K(\mathbb{Z}/n,2)\times K(\mathbb{Z},4)$ when $\epsilon_2(n)n|m$. As will be shown later, the essential case of this section is $n=2$, which we treat first.
Consider the diagonal inclusion
$$\Delta_0:SU_2\hookrightarrow SU_{2m}.$$
Passing to the quotient spaces of the respective $\mathbb{Z}/2$ actions given by the scaler multiplication of $e^{\pi\sqrt{-1}}$, we have another inclusion
\begin{equation*}
\Delta_1:PU_2\hookrightarrow P(2,2m).
\end{equation*}
Passing to classifying spaces, we have
\begin{equation*}
\mathbf{B}\Delta_1: \mathbf{B}PU_2\hookrightarrow \mathbf{B}P(2,2m)
\end{equation*}
\begin{lemma}\label{H^3BDelta}
$\mathbf{B}\Delta_1$ induces an isomorphism on $H^3(-;\mathbb{Z})$.
\end{lemma}
\begin{proof}
By the Hurewicz theorem, it suffices to show that $\Delta_1:PU_2\hookrightarrow P(2,2m)$ induces an isomorphism of fundamental groups, which follows from the fact that the $\mathbb{Z}/2$ actions on their respective simply connected covers $SU_2$ and $SU_{2m}$ commute with the diagonal inclusion.
\end{proof}

Recall the well-known exceptional isomorphism $PU_2\cong SO_3$, from which it follows
\begin{equation*}
\mathbf{B}PU_2\cong\mathbf{B}SO_3,
\end{equation*}
and in particular,
\begin{equation*}
H^*(\mathbf{B}PU_2;\mathbb{Z}/2)\cong H^*(\mathbf{B}SO_3;\mathbb{Z}/2)\cong\mathbb{Z}/2[w_2,w_3],
\end{equation*}
where $w_2,w_3$ are the Stiefel-Whitney classes of the universal $SO_3$-bundle over $\mathbf{B}SO_3$.
\begin{lemma}\label{Sq^1 of w2}
In the setting above, we have
\begin{equation*}
\operatorname{Sq}^1(w_2)=w_3.
\end{equation*}
\end{lemma}

\begin{proof}
The first nontrivial stage of the Postnikov tower of $\mathbf{B}SO_3$ is $K(\mathbb{Z}/2,2)\simeq\mathbf{B}SO_3[2]$, and $w_2$ is the Postnikov map. It follows from $\pi_3(\mathbf{B}SO_3)=0$ that we have  $K(\mathbb{Z}/2,2)\simeq\mathbf{B}SO_3[3]$, from which it follows that $w_2$ induces an isomorphism over mod $2$ cohomology groups in dimensions less than $4$. The lemma then follows from Corollary \ref{K(Z/n,2)mod 2}.
\end{proof}
Recall that we denote the generator of $H^3(\mathbf{B}P(n,mn);\mathbb{Z})$ by $x_1'$. Also recall the element $R_2(x_1')\in H^7(\mathbf{B}P(2,2m);\mathbb{Z})$, or $R_2$ for short. Let overhead bars indicate the mod $2$ reduction of integral cohomology classes. We have the following
\begin{corollary}\label{R2BDelta}
In mod $2$ cohomology we have
$(\mathbf{B}\Delta_1)^*(\bar{R}_2)=w_2^2w_3$. In particular $(\mathbf{B}\Delta_1)^*(\bar{R}_2)$ is nontrivial.
\end{corollary}
\begin{proof}
It follows from Lemma \ref{H^3BDelta} and Lemma \ref{Sq^1 of w2} that $(\mathbf{B}\Delta_1)^*(\bar{x}_1')=w_3$ and $(\mathbf{B}\Delta_1)^*(x_1'')=w_2$, where $x_1''$ is the generator of $H^2(\mathbf{B}P(2,2m);\mathbb{Z}/2)$. The rest follows from Proposition \ref{R_n mod 2}.
\end{proof}
Recall from Proposition \ref{H^4(BP)} that $H^4(\mathbf{B}P(2,2m);\mathbb{Z})$ is generated by $e_2'$.
\begin{lemma}\label{e2}
Let $4|m$, and $(\mathbf{B}\Delta_1)^*$ be the homomorphism induced by $\mathbf{B}\Delta_1$ between integral cohomology groups. Then
\begin{equation*}
(\mathbf{B}\Delta_1)^*(e_2')=\frac{m}{4}e_2,
\end{equation*}
where $e_2$ and $e_2'$ are the generators of $H^4(\mathbf{B}PU_2;\mathbb{Z})$ and $H^4(\mathbf{B}P(2,2m);\mathbb{Z})$, respectively.
\end{lemma}
\begin{proof}
Consider the following composition
\begin{equation*}
\mathbf{B}PU_2\xrightarrow{\mathbf{B}\Delta_1}\mathbf{B}P(2,2m)\rightarrow\mathbf{B}PU_{2m},
\end{equation*}
where the second arrow is induced by the quotient homomorphism. By Theorem \ref{Cohomology of BPU} and Proposition \ref{H^4(BP)}, $H^4(-;\mathbb{Z})$ of all $3$ spaces involved are isomorphic to $\mathbb{Z}$. It follows from Lemma 8.2 of \cite{Gu}, or a simple calculation using the Cartan formula of Chern classes, that $H^4(-;\mathbb{Z})$ of the composition is multiplication by $m^2$, and from Proposition \ref{H^4(BP)} that $H^4(-;\mathbb{Z})$ of the second arrow is multiplication by $4m$. Therefore $H^4(\mathbf{B}\Delta_1;\mathbb{Z})$ is multiplication by
\begin{equation*}
m^2/4m=\frac{m}{4}.
\end{equation*}
\end{proof}

\begin{corollary}\label{m=4}
Let $m=4$ and $\bar{e}_2'$ be the mod $2$ reduction of $e_2'$. Then in the mod $2$ cohomology, we have
\begin{equation*}
(\mathbf{B}\Delta_1)^*(\bar{e}_2')=w_2^2.
\end{equation*}
\end{corollary}
\begin{proof}
Notice that
$$H^4(\mathbf{B}PU_2;\mathbb{Z}/2)\cong H^4(\mathbf{B}SO_3;\mathbb{Z}/2)$$
is generated by $w_2^2$. Therefore $w_2^2$ is the mod $2$ reduction of $e_2$. The rest follows from Lemma \ref{e2}.
\end{proof}
The author owes the following lemma to A. Bousfield.
\begin{lemma}\label{Sq^3e_2'}
Let $4|m$, and recall the definition of $\bar{e}_2'$ from Corollary \ref{m=4}. We have
\begin{equation*}
\operatorname{Sq}^3(\bar{e}_2')\neq 0.
\end{equation*}
\end{lemma}
\begin{proof}
Consider the following fiber sequence:
\begin{equation*}
\mathbf{B}P(2,2m)\rightarrow\mathbf{B}PU_{2m}\rightarrow K(\mathbb{Z}/m,2),
\end{equation*}
and denote its associate Serre spectral sequence in $\mathbb{Z}/2$ coefficients by $^2E_*^{*,*}$. Recall that $H^2(\mathbf{B}P(2,2m);\mathbb{Z}/2)$ is generated by $x_1''$, and $H^3(\mathbf{B}P(2,2m);\mathbb{Z}/2)$ by $\bar{x_1}'$. Also recall from Corollary \ref{K(Z/n,2)mod 2} which asserts
\begin{equation*}
H^*(K(\mathbb{Z}/n,2);\mathbb{Z}/2)=\mathbb{Z}/2[b_2, b_3, b_5]
\end{equation*}
where $b_5=\operatorname{Sq}^2(b_3)$ and $\operatorname{Sq}^1(b_2)=0$. For obvious degree reasons, the differential
\begin{equation}\label{^2d_3(0,2)}
^2d_3^{0,2}: {^2E}_3^{0,2}\cong H^2(\mathbf{B}P(2,2m);\mathbb{Z}/2)\rightarrow  {^2E}_3^{3,0}\cong H^3(K(\mathbb{Z}/m,2);\mathbb{Z}/2), x_1''\mapsto b_3
\end{equation}
is an isomorphism. Therefore, the element $x_1''$ is transgressive and we have
\begin{equation*}
\begin{split}
^2d_4^{0,3}: {^2E}_4^{0,3}\cong H^3(\mathbf{B}P(2,2m);\mathbb{Z}/2)&\rightarrow  {^2E}_4^{4,0}\cong H^4(K(\mathbb{Z}/m,2);\mathbb{Z}/2),\\
\bar{x}_1'=\operatorname{Sq}^1(x_1'')&\mapsto \operatorname{Sq}^1(b_3)=0,
\end{split}
\end{equation*}
because $b_3$ is the reduction of an integral class. From this it follows that
\begin{equation}\label{^2E(2,3)}
^2d_4^{2,3}(b_2\otimes\bar{x}_1')=0.
\end{equation}
It follows from Proposition \ref{kappa3} that
\begin{equation*}
H^4(\mathbf{B}P(2,2m);\mathbb{Z}/2)\cong\mathbb{Z}/2\oplus\mathbb{Z}/2,
\end{equation*}
and moreover, it is generated by $(x_1'')^2$ and $\bar{e}_2'$. Since $x_1''\in {^2E}_{3}^{0,2}$ is transgressive, so is $(x_1'')^2=\operatorname{Sq}^2(x_1'')$. Furthermore, we have
\begin{equation}\label{^2d_5(0,4)}
^2d_5^{0,4}((x_1'')^2)=\operatorname{Sq}^2(^2d_5^{0,4}(x_1''))=\operatorname{Sq}^2(b_3)=b_5.
\end{equation}
For obvious degree reasons, this is the only nontrivial differential reaching $^2E_*^{5,0}$. Hence, we have
\begin{equation}\label{^2E(5,0)}
^2E_{\infty}^{5,0}\cong\mathbb{Z}/2.
\end{equation}

On the other hand, it follows from (\ref{^2E(2,3)}) that $b_2\otimes\bar{x}_1'\in ^2E_3^{2,3}$ is a permanent cocycle, whereas
\begin{equation*}
H^5(\mathbf{B}PU_{2m};\mathbb{Z}/2)\cong\mathbb{Z}/2,
\end{equation*}
a consequence of Theorem \ref{Cohomology of BPU}. Therefore, it follows from (\ref{^2E(5,0)}) that $b_2\otimes\bar{x}_1'$ is a coboundary. For degree reasons and (\ref{^2d_5(0,4)}), we have
\begin{equation}\label{^2d_2(0,4)}
^2d_2^{0,4}(\bar{e}_2')=b_2\otimes\bar{x}_1'.
\end{equation}
We recall a theorem regarding Steenrod operations in spectral sequences, proved independently by Araki (\cite{Ar}) and V{\'a}zquez (\cite{Va}). We quote this theorem from \cite{Mc} as follows:
\begin{theorem}[Theorem 6.15, \cite{Mc}]\label{Steenrod in E}
On the mod $p$ cohomology spectral sequence associated to a fibration $F\rightarrow E\rightarrow B$, there are operations
\begin{equation*}
\textrm{for $p$ odd}
\begin{cases}
_F\operatorname{P}^s: E_r^{a,b}\rightarrow E_r^{a,b+2s(p-1)}, 1\leq r\leq\infty,\\
_B\operatorname{P}^s: E_r^{a,b}\rightarrow E_r^{a+(2s-b)(p-1),pb}, 2\leq r\leq\infty,
\end{cases}
\end{equation*}

\begin{equation*}
\textrm{for $p=2$}
\begin{cases}
_F\operatorname{Sq}^i: E_r^{a,b}\rightarrow E_r^{a,b+i}, 1\leq r\leq\infty,\\
_B\operatorname{Sq}^i: E_r^{a,b}\rightarrow E_r^{a+i-b,2b}, 2\leq r\leq\infty,
\end{cases}
\end{equation*}
that converge to the action of $\mathscr{A}_p$ on $H^*(E;\mathbb{Z}/p)$, commute with the differentials
in the spectral sequence, satisfy analogues of Cartan's formula and the Adem relations and reduce to the $\mathscr{A}_p$-action on $H^*(F;\mathbb{Z}/p)$ and $H^*(B;\mathbb{Z}/p)$, when $r=2$ and $a=0$ or $b=0$ (that is, on $E_2^{*,0}$ and $E_2^{0,*}$). Here $\mathscr{A}_p$ denotes the mod $p$ Steenrod algebra.
\end{theorem}
These operations satisfy a list of axioms similar to those characterizing Steenrod operations. In particular, we have
\begin{equation}\label{Araki}
^F\operatorname{Sq}^i=0: E_r^{a,b}\rightarrow E_r^{a,b+i}, i<0 \textrm{ or }i>b.
\end{equation}
For the complete list of the axioms, see, for example, \cite{Ar}.

It follows from (\ref{^2d_2(0,4)}) that
\begin{equation*}
^2d_2^{0,7}(\operatorname{Sq}^3(\bar{e}_2'))=_F\operatorname{Sq}^3(b_2\otimes\bar{x}_1')=b_2\otimes(\bar{x}_1')^2\neq 0.
\end{equation*}
In particular, $\operatorname{Sq}^3(\bar{e}_2')\neq 0$.
\end{proof}

In order to reduce the proof of Theorem \ref{main2} to the case that $n=2$, we need the following
\begin{theorem}[Theorem 1.3, \cite{An3}]\label{separate primes}
Let $(X,\mathscr{O}_X)$ be a connected locally ringed topos, and let $\alpha=\alpha_1+\cdots +\alpha_t$ be the prime decomposition of a Brauer class $\alpha\in\operatorname{Br}_{top}(X)$ so that each $\operatorname{per}(\alpha_i)=p_i^{a_i}$ for distinct primes $p_1,\cdots, p_t$. Then
$$\operatorname{ind}(\alpha)=\operatorname{ind}(\alpha_1)\cdots\operatorname{ind}(\alpha_t).$$
\end{theorem}

\begin{proof}[Proof of Theorem \ref{main2}]
Suppose $n=2l$ where $l$ is an odd number. Write $\alpha=\alpha_1+\alpha_2$, where $\alpha_1$ and $\alpha_2$ are of order $2$ and $l$ respectively. It follows from Theorem \ref{main1} that $\operatorname{ind}(\alpha_2)|\epsilon_3(l)l^3$. By Theorem \ref{separate primes}, it suffices to show that $\operatorname{ind}(\alpha_1)\nmid 2^3$. Hence, it suffices to prove the theorem for $n=2$.

Recall that for any $n$
\begin{equation*}
\kappa_5=\lambda_1R_n\times 1+ \lambda_2\beta_n\times\iota_4+1\times\Gamma_4.
\end{equation*}
In the case $n=2$, this means that $\lambda_1$ and $\lambda_2$ are either $0$ or $1$. Therefore, it suffices to determine $\lambda_1$ and $\lambda_2$ in the mod $2$ cohomology group. It follows from Corollary \ref{R2BDelta} and Corollary \ref{m=4} that in $H^7(\mathbf{B}PU_2;\mathbb{Z}/2)$, we have
\begin{equation*}
\begin{split}
0=&(\mathbf{B}\Delta_1)^*(\lambda_1 \bar{R}_2+\lambda_2\bar{x}_1'\bar{e}_2'+\operatorname{Sq}^3(\bar{e}_2'))\\
=&(\mathbf{B}\Delta_1)^*(\lambda_1 \bar{R}_2)+\lambda_2 w_2^2w_3+\operatorname{Sq}^3(w_2^2)\\
=&\lambda_1w_2^2w_3+\lambda_2 w_2^2w_3+[\operatorname{Sq}^2(w_2)w_3+w_3\operatorname{Sq}^2(w_2)]\\
=&(\lambda_1+\lambda_2)w_2^2w_3
\end{split}
\end{equation*}
which implies $\lambda_1+\lambda_2=0$. On the other hand, it follows from Lemma \ref{Sq^3e_2'} that $\lambda_1$ and $\lambda_2$ cannot be both $0$. So we have
$$\lambda_1=\lambda_2=1.$$
Therefore, when $m=4$, we have
\begin{equation*}
\kappa_5=R_2\times 1+ \beta_2\times\iota_4+1\times\Gamma_4.
\end{equation*}
Hence, the obstruction class for lifting $\beta_2$ to $\mathbf{B}P(2,8)[5]$ is $R_2\neq 0$, and the desired result follows.
\end{proof}

\bibliographystyle{abbrv}
\bibliography{ref}

\begin{thebibliography}{10}

\bibitem{An2}
B.~Antieau and B.~Williams.
\newblock The topological period-index conjecture.
\newblock {\em unpublished}.

\bibitem{An}
B.~Antieau and B.~Williams.
\newblock The topological period--index problem over $6$-complexes.
\newblock {\em Journal of Topology}, 7(3):617--640, 2013.

\bibitem{An1}
B.~Antieau and B.~Williams.
\newblock The period-index problem for twisted topological {$K$}-theory.
\newblock {\em Geometry \& Topology}, 18(2):1115--1148, 2014.

\bibitem{An4}
B.~Antieau and B.~Williams.
\newblock The prime divisors of the period and index of a {B}rauer class.
\newblock {\em Journal of Pure and Applied Algebra}, 219(6):2218--2224, 2015.

\bibitem{An3}
B.~Antieau and B.~Williams.
\newblock Prime decomposition for the index of a brauer class.
\newblock {\em Annali della Scuola Normale Superiore di Pisa. Classe di
  scienze}, 17(1):277--285, 2017.

\bibitem{Ar}
S.~Araki.
\newblock Steenrod reduced powers in the spectral sequences associated with a
  fibering.
\newblock {\em Memoirs of the Faculty of Science, Kyushu University. Series A,
  Mathematics}, 11(1):15--64, 1957.

\bibitem{At2}
M.~Atiyah.
\newblock {\em {$K$}-theory}.
\newblock CRC Press, 2018.

\bibitem{At}
M.~Atiyah and G.~Segal.
\newblock Twisted {$K$}-theory.
\newblock {\em Ukrainian Mathematical Bulletin}, 1(3):287--330, 2004.

\bibitem{At1}
M.~Atiyah and G.~Segal.
\newblock Twisted {$K$}-theory and cohomology.
\newblock {\em Inspired by S. S. Chern, Nankai Tracts Math.}, 11:5--43, 2006.

\bibitem{Co}
J.-L. Colliot-Th{\'e}lene.
\newblock Exposant et indice d'alg{\`e}bres simples centrales non
  ramifi{\'e}es.
\newblock {\em ENSEIGNEMENT MATHEMATIQUE}, 48(1/2):127--146, 2002.

\bibitem{Do}
P.~Donovan and M.~Karoubi.
\newblock Graded {B}rauer groups and {$K$}-theory with local coefficients.
\newblock {\em Publications Math\'ematiques de l'IH\'ES}, 38:5--25, 1970.

\bibitem{Gr}
A.~Grothendieck.
\newblock Le groupe de {B}rauer. i. algebres d¡¯azumaya et interpr{\'e}tations
  diverses.
\newblock {\em Dix expos{\'e}s sur la cohomologie des sch{\'e}mas}, 3:46--66,
  1968.

\bibitem{Gu}
X.~Gu.
\newblock On the cohomology of the classifying spaces of projective unitary
  groups.
\newblock {\em J. Topol. Anal}, to appear, arXiv:1612.00506, 2016.

\bibitem{Mc}
J.~McCleary.
\newblock {\em A user's guide to spectral sequences}.
\newblock Number~58. Cambridge University Press, 2001.

\bibitem{Mo}
R.~E. Mosher and M.~C. Tangora.
\newblock {\em Cohomology operations and applications in homotopy theory}.
\newblock Courier Corporation, 2008.

\bibitem{Ca}
{\'E}.~normale~sup{\'e}rieure (France) and H.~Cartan.
\newblock {\em S{\'e}minaire {Henri Cartan}: ann. 7 1954/1955; Alg{\`e}bres
  {d'Eilenberg-Maclane} et homotopie}.
\newblock Secretariat Mathematique, 1958.

\bibitem{Va}
R.~Vasquez.
\newblock Nota sobre los cuadrados de {S}teenrod en la sucesion espectral de un
  espacio fibrado.
\newblock {\em Bol. Soc. Mat. Mexicana}, 2:1--8, 1957.

\bibitem{We}
C.~A. Weibel.
\newblock {\em An introduction to homological algebra}.
\newblock Number~38. Cambridge university press, 1995.

\end{thebibliography}
\end{document}